\documentclass[12pt, notitlepage]{amsart}
\usepackage{latexsym, amsfonts, amsmath, amssymb, amsthm, cite, tabls, paralist}

\newtheorem{theorem}{Theorem}[section]
\newtheorem{lemma}[theorem]{Lemma}
\newtheorem{proposition}[theorem]{Proposition}

\newtheorem{corollary}[theorem]{Corollary}

\theoremstyle{remark}
\newtheorem{remark}[theorem]{Remark}

\theoremstyle{definition}
\newtheorem{definition}[theorem]{Definition}

\theoremstyle{remark}
\newtheorem{example}[theorem]{Example}

\newtheorem{case}{Case}

\numberwithin{equation}{section}

\newcommand{\Z}{\mathbb{Z}}

\newcommand{\R}{\mathbb{R}}
\newcommand{\prob}{\mathbf{P}}

\newcommand{\E}{\mathbb{E}}

\newcommand{\ve}{\mathbf}

\newcommand{\CF}{\mathcal{F}}
\newcommand{\CP}{\mathcal{P}}
\newcommand{\CH}{\mathcal{H}}

\newcommand{\singular}{\mathcal{G}}
\newcommand{\codim}{\text{codim }}

\usepackage{graphicx}
\usepackage{mathrsfs}

\usepackage{varioref}

\begin{document}

\title{Narrow arithmetic progressions in the primes}
\author{Xuancheng Shao}
\address{Mathematical Institute\\ Radcliffe Observatory Quarter\\ Woodstock Road\\ Oxford OX2 6GG \\ United Kingdom}
\email{Xuancheng.Shao@maths.ox.ac.uk}
\thanks{XS is supported by a Glasstone Research Fellowship.}

\maketitle

\begin{abstract}
We study arithmetic progressions in primes with common differences as small as possible. Tao and Ziegler showed that, for any $k \geq 3$ and $N$ large, there exist non-trivial $k$-term arithmetic progressions in (any positive density subset of) the primes up to $N$ with common difference $O((\log N)^{L_k})$, for an unspecified constant $L_k$. In this work we obtain this statement with the precise value $L_k = (k-1) 2^{k-2}$. This is achieved by proving a relative version of Szemer\'{e}di's theorem for narrow progressions requiring simpler pseudorandomness hypotheses in the spirit of recent work of Conlon, Fox, and Zhao.
\end{abstract}

\section{Introduction}

A central problem in additive number theory concerns finding in the set of primes various linear patterns, such as $k$-term arithmetic progressions ($k$-APs) for $k\geq 2$. The groundbreaking work of Green and Tao \cite{GT08} shows that any positive density subset of the primes contains infinitely many $k$-APs.

\begin{theorem}[Arithmetic progressions in primes]\label{thm:gt}
Let $k\geq 2$ be a positive integer and $\delta>0$ be real. Let $N$ be sufficiently large depending on $k$ and $\delta$. Then any subset $A\subset\mathcal{P}\cap [N]$ with $|A|\geq \delta N/\log N$ contains a nontrivial $k$-AP.
\end{theorem}

Here $\CP$ denotes the set of primes, $[N]$ denotes the interval $\{1,2,\cdots,N\}$, and a $k$-AP is called nontrivial if its common difference is nonzero. Recall Szemer\'{e}di's theorem, which asserts the existence of $k$-APs in dense subsets of the integers. Since the set of primes has density zero in the integers, Szemer\'{e}di's theorem does not immediately imply Theorem \ref{thm:gt}. The main idea in \cite{GT08}, now referred to as the \textit{transference principle}, is then to place the set of primes densely inside a superset of ``almost primes", and to show that this superset satisfies certain pseudorandomness hypotheses so that it behaves just like the set of all integers.

\begin{theorem}[Relative Szemer\'{e}di's theorem]\label{thm:gt-relative}
Let $k\geq 2$ be a positive integer and $\delta>0$ be real. Let $N$ be prime and sufficiently large depending on $k$ and $\delta$. Let $G=\Z/N\Z$ and let $f,\nu:G\rightarrow \R$ be functions satisfying $0 \leq f\leq \nu$. Suppose that $\nu$ satisfies the $k$-linear forms conditions, and that $\E f \geq \delta$. Then $\Lambda(f,\cdots,f) \geq c$ for some constant $c=c(k,\delta)>0$.
\end{theorem}

Here $\E f$ and $\E \nu$ denotes the average value of $f$ and $\nu$, respectively, and the counting function $\Lambda(f_1,\cdots,f_k)$ is defined by
\[ \Lambda(f_1,\cdots,f_k) = \E_{n\in G} \E_{d\in G} f_1(n) f_2(n+d) \cdots f_k(n+(k-1)d) \]
for functions $f_1,\cdots,f_k: G\rightarrow\R$.

Recently Conlon-Fox-Zhao \cite{CFZ13} found a simpler proof of Theorem \ref{thm:gt-relative} using a sparse hypergraph regularity lemma, which also has the pleasant consequence of weakening the linear forms conditions that the majorant $\nu$ must satisfy. For the precise definition of these linear forms conditions, see Definition \ref{def:linear-forms} below and the remarks following it.

The main goal of this paper is to find $k$-APs in primes with common difference as small as possible. This problem of finding narrow progressions in the primes has been studied by Tao and Ziegler \cite{TZ08,TZ14}. In fact, they studied the much more general problem of finding narrow \textit{polynomial} progressions of the form $a+P_1(d),\cdots,a+P_k(d)$, where $P_1,\cdots,P_k$ are polynomials satisfying $P_1(0)=\cdots=P_k(0)=0$, and showed that the step of these progressions $d$ can be taken $O((\log N)^L)$ for some constant $L>0$ (depending only on $P_1,\cdots,P_k$). Moreover, they remarked that, in the case of arithmetic progressions, $L$ can be taken to be $Ck2^k$ for some absolute constant $C>0$ by following their arguments specialized to APs. Our main result confirms this remark, and moreover gives a precise value of the exponent $L$, which we will argue is optimal under current technologies.

\begin{theorem}[Narrow arithmetic progressions in primes]\label{thm:narrow}
Let $k\geq 2$ be a positive integer and $\delta>0$ be real. Let $N$ be sufficiently large depending on $k$ and $\delta$. Then any subset $A\subset\mathcal{P}\cap [N]$ with $|A|\geq \delta N/\log N$ contains a nontrivial $k$-AP with common difference $d$ satisfying $|d| = O_{k,\delta}( (\log N)^{L_k} )$ for any $\varepsilon>0$, where $L_k=(k-1)2^{k-2}$.
\end{theorem}

Just as Theorem \ref{thm:gt} is deduced from Theorem \ref{thm:gt-relative}, Theorem \ref{thm:narrow} will be deduced from the following relative version for narrow progressions.

\begin{theorem}[Relative Szemer\'{e}di's theorem for narrow progressions]\label{thm:relative-szemeredi}
Let $k\geq 2$ be a positive integer and $\delta>0$ be real.  Let $N$ be prime and sufficiently large depending on $k$ and $\delta$. Let $G=\Z/N\Z$ and let $f,\nu:G\rightarrow \R$ be functions satisfying $0 \leq f\leq \nu$. Let $D, S \geq 2$ be positive integers satisfying $S = o(D)$. Suppose that $\nu$ satisfies the $k$-linear forms conditions with width $S$, and that $\E f \geq \delta$. Then $\Lambda_D(f,\cdots,f) \geq c$ for some constant $c=c(k,\delta)>0$.
\end{theorem}

Here the counting function $\Lambda_D(f_1,\cdots,f_k)$ is defined by
\[  \Lambda_D(f_1,\cdots,f_k)=\E_{n\in G}\E_{d\in [D]} f_1(n)f_2(g+d)\cdots f_{k}(n+(k-1)d) \]
for $f_1,\cdots,f_k: G\rightarrow \R$, and the interval $[D]$ is embedded in $G$ in the obvious way.

See Definition \ref{def:linear-forms} below for the precise definition of the $k$-linear forms conditions with width $S$, which are analogues of the $k$-linear forms conditions needed in Conlon-Fox-Zhao's work \cite{CFZ13} in the narrow setting. 


\begin{remark}[The exponent $L_k$]
If the set $\CP$ in Theorem \ref{thm:narrow} is replaced by a random subset of $[N]$ with density $1/\log N$, then the statement holds with $L_k$ replaced by $k-1$ almost surely (see \cite[Proposition 2]{TZ14}). On the other hand, Theorem \ref{thm:narrow} fails if the exponent is smaller than $k-1$ (see \cite[Proposition 1]{TZ14}). In Remark \ref{rem:Ck} below, we will see that (the normalized characteristic function of) a random subset of $[N]$ with density $\alpha$ satisfies the $k$-linear forms conditions with width $O(\alpha^{-L_k})$, and moreover the exponent $L_k$ here is optimal. Thus if one tries to prove Theorem \ref{thm:narrow} with a smaller value of $L_k$ via a transference principle, it is necessary to seek for even more simplified linear forms conditions than those in \cite{CFZ13}.
\end{remark}

One might ultimately be interested in the case when $A = \CP$ is the set of all primes. The Hardy-Littlewood conjecture implies that there are infinitely many nontrivial $k$-APs in primes with common difference $O_k(1)$. This is only known unconditionally in the case $k=2$ thanks to recent breakthroughs by Zhang \cite{Zhang14} and by Maynard \cite{May15} (and by Tao independently), which asserts that there are infinitely many pairs of primes with bounded gap. The Hardy-Littlewood conjecture also predicts an asymptotic formula for the number of $k$-APs in primes up to $N$ of a given common difference. For the problem of counting all $k$-APs in primes (without any restrictions on $d$), and indeed for counting any linear pattern with \textit{finite complexity}, such an asymptotic formula is established in \cite{GT10} (with a crucial ingredient in \cite{GTZ12}). Finally, one could also ask for asymptotic formulas of this type with the Liouville function $\lambda$ or the Mobi\"{u}s function $\mu$ (in which case the main term should be zero). Strong results of this type are recently established  by Matom\"{a}ki-Radziwi{\l}{\l}-Tao \cite{MRT15} in the case $k=2$. They showed that
\[ \E_{d\in [D]} \left| \E_{n \in [N]} \mu(n) \mu(n+d) \right| = o(1) \]
as soon as $D\rightarrow\infty$, with a crucial input from \cite{MR15} regarding multiplicative functions in (very) short intervals.

\section{Outline of proof}\label{sec:outline}

{\em Conventions.} Throughout this paper we fix the positive integer $k \geq 2$. We always work in the cyclic group $G = \Z/N\Z$, where $N$ is always assumed to be prime and sufficiently large. An integer $n$ is also viewed as an element in $G$ in the natural way. We use $o(1)$ to denote a quantity that tends to zero as $N \rightarrow \infty$. For a vector $\ve{s}$, we always use $s_1, s_2,\cdots$ to denote its coordinates. Similarly, a vector $\ve{s}^{(\tau)}$ for some $\tau \in \{0,1\}$ has coordinates $s_1^{(\tau)}, s_2^{(\tau)}, \cdots$, and a vector $\ve{s}^{(\omega)}$ for some $\omega = (\omega_1,\omega_2,\cdots)$ has coordinates $s_1^{(\omega_1)}, s_2^{(\omega_2)}, \cdots$.

In this section we state the main ingredients in the proof of Theorems \ref{thm:narrow} and \ref{thm:relative-szemeredi}. We start by defining the $k$-linear forms conditions appearing in the statement of Theorem \ref{thm:relative-szemeredi} (compare with \cite[Definition 2.2]{CFZ13}).

\begin{definition}[Linear forms conditions]\label{def:linear-forms}
Let $k\geq 2$ be a positive integer. Let $N$ be prime and let $G = \Z/N\Z$. Let $S \geq 2$ be real. We say that a function $\nu:G\rightarrow\R$ satisfies the $k$-linear forms conditions with width $S$ if the following conditions hold.
\begin{enumerate}
\item For any convex body $\Omega \subset \R^{2k}$ with inradius $r(\Omega) \geq S$ and $\Omega \subset [-r(\Omega)^{O(1)}, r(\Omega)^{O(1)}]^{2k}$, we have
\[ \E_{n \in G}\E_{(\ve{s}^{(0)}, \ve{s}^{(1)}) \in \Omega \cap \Z^{2k}} \prod_{j=1}^k \prod_{\ve{\omega}\in \{0,1\}^{[k]\setminus\{j\}}} \nu \left( n+ \psi_j(\ve{s}^{(\ve{\omega})}) \right)^{e(j,\ve{\omega})} =1+o(1), \]
for each choice of $e(j,\ve{\omega})\in \{0,1\}$, where $\psi_j:\Z^k\rightarrow\Z$ is the linear form defined by
\begin{equation}\label{eq:psi-j} 
\psi_j(s_1,\cdots,s_k)=\sum_{i=1}^k (j-i) s_i. 
\end{equation}
\item For any convex body $\Omega \subset \R^{2k}$ with inradius $r(\Omega) \geq S$ and $\Omega \subset [-r(\Omega)^{O(1)}, r(\Omega)^{O(1)}]^{2k}$, we have
\[ \E_{n \in G} \E_{(\ve{s}^{(0)}, \ve{s}^{(1)}) \in \Omega \cap \Z^{2k}}  \prod_{\omega \in \{0,1\}^{k}} \nu\left( n + \psi(\ve{s}^{(\omega)}) \right)^{e(\omega)} = 1 + o(1), \]
for each choice of $e(\omega) \in \{0,1\}$, where $\psi:\Z^k\rightarrow\Z$ is the linear form defined by
\begin{equation}\label{eq:psi} 
\psi(s_1,\cdots,s_k) = k! \sum_{i=1}^k s_i.
\end{equation}
\item For any convex body $\Omega \subset \R^2$ with inradius $r(\Omega) \geq S$ and $\Omega \subset [-r(\Omega)^{O(1)}, r(\Omega)^{O(1)}]^{2}$, and any $1 \leq j \leq k$, we have
\[ \E_{n \in G} \nu(n)^{e} \sum_{(d^{(0)}, d^{(1)}) \in \Omega \cap \Z^2} \prod_{\substack{1 \leq i \leq k\\ i \neq j}} \prod_{\tau \in \{0,1\}} \nu(n + (i-j) d^{(\tau)})^{e(i,\tau)} = 1 + o(1), \]
for each choice of $e, e(i,\tau) \in \{0,1\}$.
\end{enumerate}
\end{definition}

In the first condition, since $\psi_j$ does not depend on the $j$th variable, $\psi_j(\ve{s}^{(\omega)})$ makes sense for $\omega\in\{0,1\}^{[k]\setminus\{j\}}$. As explained in \cite[Section 2.2]{CFZ13}, the first set of these linear forms conditions occur quite naturally, corresponding to $2$-blowups of triangles in appropriate hypergraphs.  These blowups are eventually responsible for the extra factor of $2^{k-2}$ in the exponent $L_k$. The presence of the other linear forms conditions are purely technical, coming from  extra manoeuvres required to deal with the narrow nature of the progressions. However, the value of $L_k$ depends critically on only the first set of conditions. 

\begin{example}\label{eg:k=3}
When $k=3$, the first condition in the $3$-linear forms conditions are saying that the product of the following $12$ terms:
\[  \begin{matrix} \nu(n-x_2-2x_3), & \nu(n-x_2-2y_3), & \nu(n-y_2-2x_3), & \nu(n-y_2-2y_3), \\
 \nu(n+x_1-x_3), & \nu(n+y_1-x_3), & \nu(n+x_1-y_3), & \nu(n+y_1-y_3), \\
 \nu(n+2x_1+x_2), & \nu(n+2y_1+x_2), & \nu(n+2x_1+y_2), & \nu(n+2y_1+y_2),
 \end{matrix} \] 
when averaged over $n\in G$ and $(x_1,x_2,x_3,y_1,y_2,y_3) \in \Omega \cap \Z^6$, is equal to $1+o(1)$. The same holds for the product of any subset of these $12$ terms.
\end{example}

The proof of the relative Szemer\'{e}di theorem for narrow progressions (Theorem \ref{thm:relative-szemeredi}) will be carried out in Sections \ref{sec:relative}-\ref{sec:counting-lemma}. While the proof of its global analogue (Theorem \ref{thm:gt-relative}) in \cite{CFZ13} proceeds by passing to the corresponding counting problem in hypergraphs, we are unable to find a good graph model for counting narrow progressions. We thus proceed entirely in the arithmetic setting, motivated by the work of Zhao \cite{Zhao14}.

\begin{remark}\label{rem:Ck}
Now that the $k$-linear forms conditions are precisely defined, let us explain why any majorant $\nu$ for the primes should not satisfy the $k$-linear forms conditions with width $S$ below $(\log N)^{L_k}$.  We illustrate this with the example $k=3$ and $L_3=4$, and recall the linear forms in Example \ref{eg:k=3}. Consider the contribution from those terms with $x_1=y_1$. Under this restriction, four pairs of these linear forms take the same values. Since $\nu^2$ should have average about $\log N$, the average over all terms with $x_1=y_1$ should have size about $(\log N)^4$. Thus if $S$ is smaller than $(\log N)^4$, these contributions will dominate and the linear forms conditions fail. Similarly, for general $k$, the restriction $x_1=y_1$ creates $L_k=(k-1)2^{k-2}$ pairs of linear forms having the same value, and thus $S$ must be larger than $(\log N)^{L_k}$. The same argument also shows that, if $\nu$ is the normalized characteristic function of a random subset of $[N]$ with density $\alpha$, then it does not satisfy the $k$-linear forms conditions with width $S$ below $\alpha^{-L_k}$.
\end{remark}

This remark motivates the following definition.

\begin{definition}\label{def:rep}
Let $\Psi = (\psi_1,\cdots,\psi_t):\Z^d\rightarrow\Z^t$ be a system of distinct affine linear forms in $d$ variables $\ve{x}=(x_1,\cdots,x_d)$. For any $I \subset [t]$, let $\Psi_I = \{\psi_i : i \in I\}$ and define
\[ \Pi(\Psi_I) = \{ \ve{x}\in\R^d : \psi_i(\ve{x}) = \psi_j(\ve{x}) \text{ whenever } i,j \in I \}. \]
Furthermore, for any partition $\pi$ of $[t]$ (so that $\pi$ is a collection of disjoint subsets of $[t]$ whose union is $[t]$), define
\[ \Pi(\Psi, \pi) = \bigcap_{I \in \pi} \Pi(\Psi_I). \]
Finally, define
\[ L(\Psi) = \sup_{|\pi|<t} \frac{t-|\pi|}{\codim \Pi(\Psi,\pi)}, \]
where $|\pi|$ denotes the number of subsets in the partition $\pi$, and the supremum is taken over all partitions $\pi$ of $[t]$ with $|\pi| < t$.
\end{definition}

The denominator $\codim \Pi(\Psi,\pi)$ is the smallest number of independent linear conditions on $x_1,\cdots,x_d$ needed to create a linear subvariety on which linear forms from the same atom of $\pi$ are identical. By convention we set $\codim \Pi(\Psi,\pi) = \infty$ if $\Pi(\Psi,\pi) = \emptyset$. Since $\Psi$ consists of distinct linear forms, this codimension is positive whenever $|\pi| <t$. If $\Psi_k$ is the collection of linear forms appearing in the first set of $k$-linear forms conditions, then $L(\Psi_k) \geq L_k$ by Remark \ref{rem:Ck}. We will show in Section \ref{sec:linear-algebra} that equality holds.

\begin{proposition}\label{thm:rep-index}
Let $k\geq 2$ be a positive integer, and let $\Psi$ be the system of linear forms appearing in the first condition in the $k$-linear forms conditions. More precisely, $\Psi$ is the collection of linear forms $\psi$ in $2k$ variables $\ve{s}^{(0)} = (s_1^{(0)},s_2^{(0)},\cdots,s_k^{(0)})$ and $\ve{s}^{(1)} = (s_1^{(1)},s_2^{(1)},\cdots,s_k^{(1)})$ of the form
\[ \psi( \ve{s}^{(0)}, \ve{s}^{(1)} ) = \psi_j(\ve{s}^{(\omega)}) \]
for some $1\leq j\leq k$ and $\omega \in \{0,1\}^{[k]\setminus\{j\}}$, where $\psi_j$ is defined in \eqref{eq:psi-j}. Then $L(\Psi) = L_k = (k-1)2^{k-2}$.
\end{proposition}

This proposition explains the occurrence of $L_k$ in Theorem \ref{thm:narrow}. In principle we also need to evaluate $L(\Psi)$ for the systems $\Psi$ appearing in the second and the third set of $k$-linear forms conditions. These tasks are much easier. A moment's thought reveals that $L(\Psi) = 2^{k-1}$ when $\Psi$ is the system in the second set of conditions, and $L(\Psi) = k-1$ when $\Psi$ is the system in the third set of conditions. Both values are at most $(k-1)2^{k-2}$ when $k \geq 3$.

In order to apply Theorem \ref{thm:relative-szemeredi} we also need a majorant $\nu$ for the ($W$-tricked) primes satisfying the $k$-linear forms conditions. The idea of using a smoothly truncated version of Selberg's weight was first consider by Goldston-Pintz-Y{\i}ld{\i}r{\i}m \cite{GPY09,GY07}. See also \cite[Appendix D]{GT10} and the note \cite{Tao07}. In Sections \ref{sec:majorant}-\ref{sec:cor} we will review the basic properties of this majorant and prove that it satisfies the $k$-linear forms conditions.

\begin{proposition}[Pseudorandom majorants]\label{thm:majorant}
Fix a positive integer $t_0$.  Let $N$ be prime and sufficiently large, and let $G = \Z/N\Z$. Let $w \leq 0.1\log\log N$ be a slowly growing function of $N$, and let $W = \prod_{p\leq w}p$. Take any reduced residue class $b\pmod W$. There exists a function $\nu = \nu_{W,b} : G\rightarrow\R$ satisfying the following conditions.
\begin{enumerate}
\item $\nu_{W,b}(n) \geq 0$ for any $n$ and moreover
\[ \nu_{W,b}(n) \geq \frac{c\varphi(W)\log N}{W} \]
for some constant $c = c(t_0) > 0$, whenever $Wn+b$ is prime and $Wn+b > N^{1/2}$. 
\item For any system of distinct affine linear forms $\Psi = (\psi_1,\cdots,\psi_t): \Z^d\rightarrow \Z^t$ with $t \leq t_0$, and any convex body $\Omega \subset \R^d$ with inradius $r(\Omega)$ and $\Omega \subset [-r(\Omega)^{O(1)}, r(\Omega)^{O(1)}]^{d}$, such that $r(\Omega) \geq g(N)(\log N)^{L(\Psi)}$ for a function $g$ satisfying $g(N)\rightarrow\infty$ as $N\rightarrow\infty$, we have
\[ \E_{n \in G} \E_{\ve{x} \in \Omega \cap \Z^d} \prod_{i=1}^t \nu_{W,b}(n+\psi_i(\ve{x})) = 1 + o_{\Psi;N\rightarrow\infty}(1). \]
\end{enumerate}
\end{proposition}

In view of Proposition \ref{thm:rep-index} and the remark following it, this implies that the function $\nu_{W,b}$ satisfies the $k$-linear forms conditions with width $g(N)(\log N)^{L_k}$.
We now have all the ingredients needed to deduce Theorem \ref{thm:narrow}.

\begin{proof}[Deduction of Theorem \ref{thm:narrow} from 
Theorem~\ref{thm:relative-szemeredi} assuming Propositions \ref{thm:rep-index} and \ref{thm:majorant}]
We may assume that $k \geq 3$, as the statement is trivial when $k=2$. By a diagonalization argument, it suffices to prove the statement with $|d| \leq g(N) (\log N)^{L_k}$ for any slowly growing function $g$ and large $N$.  
Let $w = w(N) \leq 0.1 \log g(N)$ be a slowly growing function and let $W = \prod_{p\leq w}p$, so that $W \leq g(N)^{1/2}$.
Choose a prime $N' \in [2N/W,4N/W]$, and let $G = \Z/N'\Z$. By the pigeonhole principle, we may choose a reduced residue class $b\pmod W$ such that 
\begin{equation}\label{eq:bmodW} 
A \cap \{p\in\CP: p\equiv b\pmod W\text{ and }p>N^{1/2}\} \geq  \frac{ \delta N}{\varphi(W)\log N} - N^{1/2} \geq \frac{\delta N}{2\varphi(W)\log N}.
\end{equation}
Let $\nu = \nu_{W,b}:G\rightarrow\R$ be the majorant from Proposition~\ref{thm:majorant}, and let $f:G\rightarrow\R$ be the function defined by
\[ f(n) = \begin{cases} \frac{c\varphi(W)\log N'}{W} & Wn+b\in A\text{ and }Wn+b>N'^{1/2} \\ 0 & \text{otherwise,} \end{cases} \]
where $c = c(k) > 0$ is sufficiently small. Then $0\leq f\leq\nu$. Moreover, from \eqref{eq:bmodW} we obtain
\[ \E_{n\in G} f(n) \geq \frac{1}{N'}\cdot \frac{c\varphi(W)\log N}{W} \cdot \frac{ \delta N}{2\varphi(W)\log N} \geq \frac{c\delta}{10}. \]
Set $S =  g(N)^{1/4}(\log N)^{L_k}$ and $D = \lfloor g(N)^{1/2} (\log N)^{L_k}\rfloor$.  Since $\nu$ satisfies the $k$-linear forms conditions with width $S$ (see the remark following Proposition~\ref{thm:majorant}), we may apply Theorem \ref{thm:relative-szemeredi} to conclude that $\Lambda_D(f,\cdots,f) \gg_{\delta} 1$. In other words, there exist $k$-APs $n,n+d,\cdots,n+(k-1)d$ with $1\leq d\leq D$ such that each $n+jd$ ($0\leq j\leq k-1$) lies in the support of $f$. Each such $k$-AP gives rise to a $k$-AP $Wn+b,W(n+d)+b,\cdots,W(n+(k-1)d)+b$ in $A$, with step $Wd \leq WD \leq g(N) (\log N)^{L_k}$, as desired.
\end{proof}

\section{The truncated von-Mangoldt function and the prime majorant}\label{sec:majorant}

We construct the majorant $\nu$ required in Proposition \ref{thm:majorant} as follows.
Let $R \leq N^{1/2}$ be a parameter and let $\chi:\R\rightarrow\R$ be a smooth function supported on $[-1,1]$. Assume that $\chi(0)\geq 1/2$ and moreover
\begin{equation}\label{eq:c2} 
\int_{-\infty}^{\infty}|\chi'(t)|^2dt=1. 
\end{equation}
Define the truncated von-Mangoldt function $\Lambda_{\chi,R}$ with parameter $R$ and the smooth cutoff $\chi$ by the formula
\begin{equation}\label{eq:Lambda} 
\Lambda_{\chi,R}(n)=\log R\sum_{d\mid n}\mu(d)\chi\left(\frac{\log d}{\log R}\right). 
\end{equation}
Note that if $n$ is prime and $n>R$, then $\Lambda_{\chi,R}(n) = \chi(0)\log R \geq (\log R)/2$. Define the majorant $\nu_{\chi,R,W,b}: G\rightarrow\R$ by the formula
\begin{equation}\label{eq:nu} 
\nu_{\chi,R,W,b}(n)=\frac{\varphi(W)}{W\log R}\Lambda_{\chi,R}(Wn+b)^2.
\end{equation}
It is clearly non-negative, and satisfies 
\begin{equation}\label{eq:majorize} 
\nu_{\chi,R,W,b}(n) \geq \frac{\varphi(W)\log R}{4W},
\end{equation}
whenever $Wn+b$ is prime and $Wn+b>R$. The smoothly truncated nature of $\chi$ allows us to obtain precise asymptotic formulas for correlation estimates involving $\Lambda_{\chi,R}$. First we need some definitions.

\begin{definition}[Singular series]\label{def:singular}
For a vector $\ve{h} = (h_1,\cdots,h_k) \in \Z^k$ we define the singular series
\[ \singular(\ve{h})=\prod_p \left(1-\frac{1}{p}\right)^{-r}\left(1-\frac{\nu_p(\ve{h})}{p}\right), \]
where $r = \#\{h_1,\cdots,h_k\}$ and $\nu_p(\ve{h})$ is the number of residue classes modulo $p$ occupied by elements in $h_1,\cdots,h_k$. For a positive integer $W$, define also the $W$-tricked singular series
\[ \singular_W(H)=\prod_{p\nmid W} \left(1-\frac{1}{p}\right)^{-r}\left(1-\frac{\nu_p(\ve{h})}{p}\right). \]
\end{definition}

\begin{definition}[Sieve factor]\label{def:sieve-factor}
Let $\chi:\R\rightarrow\R$ be a smooth compactly supported function. For any positive integer $m$, we define the sieve factor
\[ c_{\chi,m}=\int_{\R}\cdots\int_{\R} \prod_{I\subset [m]:I\neq\emptyset} \left( \sum_{j\in I} (1+it_j) \right)^{(-1)^{|I|-1}} \prod_{j=1}^m \psi(t_j) dt_j, \]
where the function $\psi:\R\rightarrow\R$ is defined by the relation
\[ e^x\chi(x)=\int_{-\infty}^{\infty} \psi(t)e^{-ixt} dt. \]
More generally, for a vector $\ve{h}=(h_1,\cdots,h_k)\in \Z^k$, define the sieve factor
\[ c_{\chi}(\ve{h})=\prod_{h\in \{h_1,\cdots,h_k\}} c_{\chi}(m(h)), \]
where $m(h)=\#\{1\leq i\leq k:h_i=h\}$.
\end{definition}

We will not directly need the precise definition of $c_{\chi,m}$, apart from the fact that $c_{\chi,2}=1$, a consequence of the normalization \eqref{eq:c2}.

\begin{proposition}[Correlation estimates for $\Lambda_{\chi,R}$]\label{prop:correlation-Lambda}
Let $N,W$ be positive integers and let $b\pmod W$ be a reduced residue class.
Let $\Lambda_{\chi,R}$ be defined as in \eqref{eq:Lambda}. Let $\ve{h}=(h_1,h_2,\cdots,h_k)\in\Z^k$. Then
\[ \E_{n\leq N}\prod_{i=1}^k\Lambda_{\chi,R}(W(n+h_i)+b)=\left(\frac{W}{\varphi(W)}\right)^r(c_{\chi}(\ve{h})\singular_W(\ve{h})+o(E(\ve{h})))(\log R)^{k-r}+O(N^{-1}R^k(\log R)^k), \]
where $r=\#\{h_1,\cdots,h_k\}$, and
\[ E(\ve{h})=\exp\left( O\left( \sum_{p\mid \Delta(\ve{h})}\frac{1}{p} \right) \right) \]
with
\begin{equation}\label{eq:Delta} 
\Delta(\ve{h})=\prod_{\substack{1\leq i<j\leq k\\ h_i\neq h_j}}(h_i-h_j). 
\end{equation}
\end{proposition}

\begin{proof}
When $W=1$, this is exactly the main result in \cite{Tao07}. The general case follows from a straightforward adaptation of the argument there.
\end{proof}

In Section \ref{sec:average-singular} we establish some auxiliary results concerning average values of singular series, used to understand averages of $\singular_W(\ve{h})$ and $E(\ve{h})$ as $\ve{h}$ varies. In Section \ref{sec:cor}, we will then prove that the function $\nu_{\chi,R,W,b}$ satisfies the required correlations estimates in Theorem \ref{thm:majorant}.

\section{Average of the singular series}\label{sec:average-singular}

In this section, we prove an auxiliary result on the average of singular series appearing in Proposition \ref{prop:correlation-Lambda}. This is a generalization of a result of Gallagher \cite{Gal76} (see also \cite{For07}).

\begin{proposition}\label{prop:average-singular}
Let $w \geq 1$ be a parameter. For each prime $p$, let $g_p(\ve{h})=g_p(h_1,\cdots,h_t)$ be a function with $g_p\geq 1$ such that the following conditions hold:
\begin{enumerate}
\item   $g_p(\ve{h})=1+O(p^{-1})$ for any $\ve{h}\in\Z^t$;
\item $ g_p(\ve{h})=1+O(p^{-2})$ if $p\nmid \Delta(\ve{h})$, where $\Delta(\ve{h})$ is defined in \eqref{eq:Delta};
\item $g_p(\ve{h})=1$ whenever $p\leq w$.
\end{enumerate}
 Define $g:\Z^t\rightarrow\R$ by the (absolutely convergent) infinite product
\[ g(\ve{h})=\prod_p g_p(\ve{h}). \]
Let $\mathcal{H}\subset\Z^t$ be a (multi)set. Then for any $Q\geq 2$ and $\varepsilon>0$, we have
\[ \E_{\ve{h}\in\mathcal{H}}g(\ve{h}) = 1 +O(w^{-1}) + O\left( \sum_{w<q\leq Q} \frac{\mu^2(q)C^{\omega(q)}}{q} \prob_{\ve{h}\in\mathcal{H}}(q\mid\Delta(\ve{h}))\right) +O_{\varepsilon}\left(Q^{-1}\max_{\ve{h}\in\mathcal{H}} |\Delta(\ve{h})|^{\varepsilon}\right),  \]
for some constant $C=O(1)$, where $\omega(q)$ is the number of prime divisors of $q$, and $\prob_{\ve{h}\in\mathcal{H}}(q\mid\Delta(\ve{h}))$ is the probability that $q\mid \Delta(\ve{h})$ when $\ve{h}$ is chosen uniformly at random from $\mathcal{H}$.
\end{proposition}

\begin{proof}
Define a new function $g'$ by the finite product
\[ g'(\ve{h}) = \prod_{p\mid\Delta(\ve{h})} g_p(\ve{h}). \]
Since
\[ g(\ve{h}) = \prod_{p\mid \Delta(\ve{h})} g_p(\ve{h}) \cdot \prod_{p>w} \left( 1+O(p^{-2}) \right) = g'(\ve{h}) (1+O(w^{-1})), \]
it suffices to prove the proposition for $g'$. From now on we thus assume that $g_p(\ve{h})=1$ whenever $p\nmid \Delta(\ve{h})$.

For fixed $\ve{h}\in\Z^t$, define a multiplicative function $a_{\ve{h}}(q)$ supported on squarefree integers $q$ by the formula 
\[ a_{\ve{h}}(q)=\prod_{p\mid q}(g_p(\ve{h})-1). \]
Then $a_{\ve{h}}(q)$ is non-negative and vanishes unless $q\mid\Delta(\ve{h})$. Moreover,
\[ g(\ve{h})=\prod_{p\mid\Delta(\ve{h})} (1+a_{\ve{h}}(p))=\sum_{q\mid\Delta(\ve{h})} a_{\ve{h}}(q). \]
Since $a_{\ve{h}}(p)=O(p^{-1})$ by hypothesis, we have 
\[ a_{\ve{h}}(q)\leq \frac{C^{\omega(q)}}{q}  \]
for some $C = O(1)$, and thus
\[ \sum_{\substack{q\mid\Delta(\ve{h})\\ q>Q}} a_{\ve{h}}(q) \leq \sum_{\substack{q\mid\Delta(\ve{h})\\ q>Q}} \frac{\mu^2(q)C^{\omega(q)}}{q} \leq \frac{1}{Q}\sum_{q\mid\Delta(\ve{h})} \mu^2(q)C^{\omega(q)} \ll \frac{|\Delta(\ve{h})|^{\varepsilon}}{Q}, \]
where the last inequality follows from the identity
\[ \sum_{q\mid\Delta(\ve{h})} \mu^2(q)C^{\omega(q)} = \prod_{p\mid\Delta(\ve{h})} (1+C) = (1+C)^{\omega(\Delta(\ve{h}))}, \]
and the bound $\omega(\Delta(\ve{h})) = o(\log |\Delta(\ve{h})|)$. Hence,
\[ g(\ve{h}) =  \sum_{\substack{q\mid \Delta(\ve{h})\\ q\leq Q}} a_{\ve{h}}(q) + O\left( Q^{-1} |\Delta(\ve{h})|^{\varepsilon} \right).  \]
Average the above equation over $\ve{h}\in\mathcal{H}$. The $q=1$ term contributes $1$ since $a_{\ve{h}}(1)=1$ for any $\ve{h}$. If $1<q\leq w$, then $a_{\ve{n}}(q)=0$ for any $\ve{n}$.
For $w<q\leq Q$, we have
\[ \E_{\ve{h}\in\mathcal{H}} a_{\ve{h}}(q) \leq \frac{C^{\omega(q)}}{q} \prob_{\ve{h}\in\mathcal{H}} (q\mid\Delta(\ve{h})). \]
This completes the proof.
\end{proof}

We will apply this proposition twice, to deal with the main term $\singular_W(\ve{h})$ and to handle the error term $E(\ve{h})$.

\begin{example}
If $g=\singular_W$ (recall Definition \ref{def:singular}), then
\[ g_p(\ve{h})=\left(1-\frac{1}{p}\right)^{-|\{h_1,\cdots,h_t\}|}\left(1-\frac{\nu_p(h_1,\cdots,h_t)}{p}\right) \]
for $p\nmid W$, and $g_p(\ve{h})=1$ for $p\mid W$. It clearly satisfies the assumptions (1) and (3) in the statement of Proposition \ref{prop:average-singular}. 
If $p\nmid \Delta(\ve{h})$, then $|\{h_1,\cdots,h_t\}|=\nu_p(h_1,\cdots,h_t)$, and thus $g_p(\ve{h})=1+O(p^{-2})$, which verifies the assumption (2). 
\end{example}

\begin{example}
If $g=E$ (defined in the statement of Proposition \ref{prop:correlation-Lambda}), then
\[ g_p(\ve{h})=\begin{cases} C_p^{1/p} & p\mid \Delta(\ve{h}) \\ 1 & p\nmid \Delta(\ve{h}), \end{cases} \]
for some constant $C_p = O(1)$. It clearly satisfies all the assumptions in the statement of Proposition \ref{prop:average-singular} with $w=1$.
\end{example}

If the set $\mathcal{H}$ equidistributes in residue classes with modulus up to $Q$, then Proposition \ref{prop:average-singular} implies that the average of $g(\ve{h})$ over $\ve{h}\in\mathcal{H}$ is $O(1)$ for any $w$, and is $1+o(1)$ if $w\rightarrow\infty$. 

\begin{corollary}\label{cor:average-singular}
Let the notations and assumptions be as in Proposition \ref{prop:average-singular}. Suppose that for each squarefree $q\leq Q$, we have $\prob_{\ve{h}\in\CH} (q\mid \Delta(\ve{h})) \leq C^{\omega(q)}q^{-1}$ for some constant $C = O(1)$. 
Then for any $\varepsilon > 0$ we have
\[ \E_{\ve{h}\in\mathcal{H}} g(\ve{h}) = 1+O(w^{-1}\log^{O(1)} (2+w)) + O_{\varepsilon}\left(Q^{-1}\max_{\ve{h}\in\mathcal{H}} |\Delta(\ve{h})|^{\varepsilon}\right). \]
\end{corollary}

\begin{proof}
In view of Proposition \ref{prop:average-singular}, it suffices to show that
\[ \sum_{q>w} \frac{\mu^2(q)C^{\omega(q)}}{q^2} \ll \frac{\log^{O(1)}(2+w)}{w}. \]
Indeed, by Rankin's trick, this sum is bounded by
\[  \sum_q \left(\frac{q}{w}\right)^{1-1/\log w} \frac{\mu^2(q)C^{\omega(q)}}{q^2} \ll \frac{1}{w} \sum_q \frac{\mu^2(q)C^{\omega(q)}}{q^{1+1/\log w}} \ll \frac{\zeta(1+1/\log w)^{O(1)}}{w} \ll \frac{(\log w)^{O(1)}}{w},  \]
as desired.
\end{proof}

\section{Pseudorandomness of the prime majorant}\label{sec:cor}

In this section we prove Proposition~\ref{thm:majorant}, using the majorant $\nu_{\chi,R,W,b}$ constructed in \eqref{eq:nu} with $R = N^{1/4t_0}$. The lower bound on $\nu_{\chi,R,W,b}$ clearly follows from \eqref{eq:majorize}.

Now fix a system of distinct affine linear forms $\Psi = (\psi_1,\cdots,\psi_t): \Z^d\rightarrow\Z^t$ with $t \leq t_0$. Note that each $\ve{x}\in\Z^d$ induces a partition $\pi(\ve{x})$ of $[t]$, according to the values $\psi_1(\ve{x}),\cdots,\psi_t(\ve{x})$. Precisely, two indices $i,j \in [t]$ lie in the same atom of $\pi(\ve{x})$ if and only if $\psi_i(\ve{x}) = \psi_j(\ve{x})$. For each partition $\pi$ of $[t]$, let $X(\pi)$ be the set of $\ve{x}\in\Z^d$ with $\pi(\ve{x}) = \pi$. It suffices to show that
\[ \frac{1}{|\Omega\cap\Z^d|}\sum_{\ve{x}\in \Omega\cap X(\pi)} \E_{n\in G} \prod_{j=1}^t \nu_{\chi,R,W,b}(n+\psi_j(\ve{x})) = \mathbf{1}_{|\pi|=t} + o(1), \]
for each partition $\pi$ of $[t]$. For the remainder of this section, we fix the partition $\pi$ and write simply $X$ for $X(\pi)$. Let $s = |\pi|$. We may assume that $\Omega\cap X$ is nonempty. The implied constants in this section are always allowed to depend on $d,t,\Psi,\pi,X$.

From the definition \eqref{eq:nu} of $\nu_{\chi,R,W,b}$, we need to show that
\[ \frac{|\Omega\cap X|}{|\Omega\cap\Z^d|} \left( \frac{\varphi(W)}{W\log R} \right)^t \E_{\ve{x}\in \Omega\cap X} \E_{n\in G} \prod_{j=1}^t \Lambda_{\chi,R}(W(n+\psi_j(\ve{x}))+b)^2 = \mathbf{1}_{s=t} + o(1). \]
By Proposition \ref{prop:correlation-Lambda}, the inner average over $n$ above is
\[ \left( \frac{W}{\varphi(W)} \right)^s (\log R)^{2t-s} \left[ c_{\chi}(\Psi(\ve{x}),\Psi(\ve{x})) \singular_W(\Psi(\ve{x})) + o(E(\Psi(\ve{x}))) \right] + O(N^{-1}R^k(\log R)^k). \]
The last error term above is negligible by the choice of $R$. Thus we need to show that
\begin{equation}\label{eq:average-singular} 
\frac{|\Omega\cap X|}{|\Omega\cap\Z^d|} \left( \frac{\varphi(W)\log R}{W} \right)^{t-s} \E_{\ve{x}\in \Omega\cap X} \left[ c_{\chi}(\Psi(\ve{x}),\Psi(\ve{x})) \singular_W(\Psi(\ve{x})) + o(E(\Psi(\ve{x}))) \right] = \mathbf{1}_{s=t} + o(1). 
\end{equation}
It is convenient to introduce the linear variety $V \subset \R^d$ consisting of those 
vectors $\ve{x}\in\R^d$ whose induced partitions $\pi(\ve{x})$ are the same as or coarser than $\pi$, and let $\widetilde{X} = V \cap \Z^d$. In other words, $\widetilde{X}$ is the set of $\ve{x}\in\Z^d$ satisfying $\psi_i(\ve{x})=\psi_j(\ve{x})$ whenever $i,j$ lie in the same atom of $\pi$. Note that $X\subset\widetilde{X}$ always, and $\widetilde{X}=\Z^d$ when $s=t$.

\begin{lemma}\label{lem:sub-lattice}
Let $L_1\subset L_2\subset\R^d$ be two lattices (not necessarily of full rank). Let $\Omega\subset\R^d$ be a convex body with inradius $r(\Omega) \geq 2$. Then
\[ |\Omega\cap L_1| \ll_{d,L_1,L_2} r(\Omega)^{\text{dim}(L_1)-\text{dim}(L_2)} |\Omega\cap L_2|. \]
\end{lemma}

\begin{proof}
Via a linear transformation (depending only on $d$ and $L_2$), we may assume that $L_2$ is the standard lattice $\Z^{\text{dim}(L_2)}\subset\R^{\text{dim}(L_2)}$ naturally embedded in $\R^d$. By restricting to $\R^{\text{dim}(L_2)}$ we may assume that $L_2=\Z^d$. With these assumptions we may use the following covering inequality in convex geometry (see \cite[Lemma C.4]{TZ08}):
\begin{equation}\label{eq:cover} 
\E_{\ve{x}\in \Omega\cap\Z^d} f(\ve{x}) \ll \sup_{y\in\R^d} \E_{\ve{x}\in (y+[-r(\Omega),r(\Omega)])\cap\Z^d} f(\ve{x}), 
\end{equation}
applied to the function $f=1_{L_1}$.
It thus suffices to show that the probability that a random point $\ve{x}$ in a $d$-dimensional box of side lengths $2r(\Omega)$ lies in $L_1$ is $O(r(\Omega)^{-\text{codim}(L_1)})$. This is clear, since any point $\ve{x}\in L_1$ is determined by $\text{dim}(L_1)$ of its coordinates, and there are $O(r(\Omega))$ ways to choose each of these coordinates.
\end{proof}

\begin{lemma}\label{lem:X-tildeX}
Let $X$ and $\widetilde{X}$ be as above. Let $\Omega\subset\R^d$ be a convex body with inradius $r(\Omega) \geq 2$. Then
\[ |\Omega\cap X| = (1+O(r(\Omega)^{-1})) |\Omega\cap\widetilde{X}|. \]
\end{lemma}

\begin{proof}
Note that $X$ is obtained from $\widetilde{X}$ by removing a few linear subvarieties $V_1,V_2, \cdots$ from $V$. Since $X$ is non-empty, these subvarieties have codimension at least $1$ in $V$. By Lemma~\ref{lem:sub-lattice} applied to (suitable translates of) $V_i\cap\Z^d$ and $V \cap \Z^d$, we obtain
\[ |\Omega \cap V_i\cap \Z^d| \ll r(\Omega)^{-1} |\Omega \cap V \cap \Z^d| \] 
for each $i$. This gives the desired conclusion.
\end{proof}

\begin{lemma}\label{lem:equi-lattice}
Let $L\subset\R^d$ be a lattice (not necessarily of full rank). Let $\Omega\subset\R^d$ be a convex body with inradius $r(\Omega) \geq 2$. For any positive integer $q$ and any function $f:L\rightarrow\R$ satisfying $f(\ve{x}+\ve{m})=f(\ve{x})$ for any $\ve{x}\in L$ and $\ve{m}\in q L$, we have
\[ \E_{\ve{x}\in \Omega\cap L} f(\ve{x}) =  \left( 1+ O_{d,L} \left( \frac{q}{r(\Omega)} \right) \right) \E_{\ve{x}\in L/qL} f(\ve{x}). \]
\end{lemma}

\begin{proof}
By a linear change of variables, we may assume that $L\subset\R^d$ is the lattice spanned by the standard basis vectors $e_1,\cdots,e_{\dim L}$. After projecting to the first $\dim L$ coordinates, we may assume that $d=\dim L$ and $L=\Z^d$. The assertion then becomes \cite[Corollary C.3]{TZ08}.
\end{proof}

\begin{corollary}[Equidistribution in residue classes]\label{cor:equi}
Let the notations be as above. Let $\mathcal{H}=\{\Psi(\ve{x}): \ve{x}\in \Omega\cap X\}$. Then for any squarefree $q \leq r(\Omega)$, we have $\prob_{\ve{h}\in\mathcal{H}} (q\mid\Delta(\ve{h})) \leq C^{\omega(q)}q^{-1}$ for some constant $C=C(\Psi,X)>0$. 
\end{corollary}

\begin{proof}
We may assume that $q$ is sufficiently large. For each $\ve{r}\in (\Z/q\Z)^d$, let $\widetilde{X}(q,\ve{r})\subset\widetilde{X}$ be the sublattice consisting of those $\ve{x}\in\widetilde{X}$ with $\ve{x}\equiv\ve{r}\pmod q$. 
Let $\ve{R}_q$ be the set of $\ve{r}\in (\Z/q\Z)^d$ satisfying 
\[ \prod_{i,j} (\psi_i(\ve{r}) - \psi_j(\ve{r})) \equiv 0\pmod q, \]
where the product is taken over all pairs $(i,j)$ such that $i,j$ lie in different atoms of the partition $\pi$. Thus $q \mid \Delta(\ve{h})$ if and only if $\ve{h} = \Psi(\ve{x})$ for some $\ve{x} \in \Omega 
\cap X$ with $\ve{x} \pmod q \in \ve{R}_q$. It suffices to show that
\[ \sum_{\ve{r}\in\ve{R}_q} |\Omega\cap\widetilde{X}(q,\ve{r})| \leq \frac{C^{\omega(q)}}{q} |\Omega\cap X|.  \]
By Lemma \ref{lem:equi-lattice} applied to (a suitable translate of) $\widetilde{X}$ and the function $f(\ve{x}) = 1_{\ve{x} \equiv \ve{r}\pmod q}$, we have
\[ \frac{|\Omega\cap \widetilde{X}(q,\ve{r})|}{|\Omega\cap\widetilde{X}|} \ll \prob_{\ve{x} \in \widetilde{X}} (\ve{x} \equiv \ve{r}\pmod q) \ll q^{-d}, \]
where the second inequality holds since $q$ is sufficiently large depending on $\widetilde{X}$. Combining this with Lemma \ref{lem:X-tildeX} we obtain
\[ |\Omega\cap \widetilde{X}(q,\ve{r})| = q^{-d} |\Omega\cap X|. \]
It thus suffices to show that $|\ve{R}_q| \leq C^{\omega(q)} q^{d-1}$. When $q$ is prime, $\ve{R}_q$ is the union of at most $s^2$ hyperplanes in $(\Z/q\Z)^d$ cut out by equations of the form $\psi_i\equiv\psi_j\pmod q$. The desired bound $|\ve{R}_q| \ll q^{d-1}$ follows in this case, since each such hyperplane contains $O(q^{d-1})$ points (recall that the implied constants here are allowed to depend on $\Psi$). For general squarefree $q$, the conclusion follows by multiplicativity.
\end{proof}

With these lemmas in hand, we may now prove \eqref{eq:average-singular} and thus complete the proof of Proposition~\ref{thm:majorant}. Since $\Omega \subset [-r(\Omega)^{O(1)}, r(\Omega)^{O(1)}]^d$, we have
\[ \max_{\ve{x} \in \Omega \cap \Z^d} |\Delta(\Psi(\ve{x}))| \ll r(\Omega)^{O(1)}. \]
In view of Corollary \ref{cor:equi}, we may apply Corollary \ref{cor:average-singular} with $Q=r(\Omega)^{0.1}$ (say) to obtain
\[ \E_{\ve{x}\in \Omega\cap X} \singular_W(\Psi(\ve{x})) = 1+o(1), \]
and
 \[ \E_{\ve{x}\in \Omega\cap X} E(\Psi(\ve{x})) = O(1). \]
To prove \eqref{eq:average-singular}, we divide into two cases according to whether $s=t$ or $s<t$. If $s=t$, then 
\[ |\Omega\cap X| = (1+O(r(\Omega)^{-1})) |\Omega\cap\Z^d| \]
by Lemma~\ref{lem:X-tildeX}. Since the values of $\psi_i(\ve{x})$ are all distinct for $\ve{x}\in X$ in this case, the sieve factor $c_{\chi}(\Psi(\ve{x}), \Psi(\ve{x}))$ is the product of copies of $c_{\chi}(2)$, and hence equal to $1$. Thus the left side of \eqref{eq:average-singular} is
\[ (1+O(r(\Omega)^{-1})) \left[ \E_{\ve{x}\in\Omega\cap X} \singular_W(\Psi(\ve{x})) + o \left( \E_{\ve{x}\in \Omega\cap X} E(\Psi(\ve{x})) \right) \right] = 1+o(1), \]
as desired. In the case when $s<t$, by Lemma \ref{lem:sub-lattice} the left side of \eqref{eq:average-singular} is bounded by
\[ r(\Omega)^{-\text{codim}(X)} (\log R)^{t-s} = o(1) \]
by the hypothesis $r(\Omega) \geq g(N)(\log N)^{L(\Psi)}$ and Definition \ref{def:rep}. This completes the proof.

\section{Determining the constant $L(\Psi)$}\label{sec:linear-algebra}

In this section we prove Proposition \ref{thm:rep-index}. It will be convenient here to parametrize the linear forms in $\Psi$ differently in the following way. For $v\in\{1,2,\cdots,k\}$ and $I\subset \{1,2,\cdots,k\}$, define
\[ \psi_{v,I}(\ve{x},\ve{y}) = \sum_{i\in I} (v-i)x_i + \sum_{i\notin I} (v-i)y_i, \]
for $\ve{x} = (x_1,\cdots,x_k) \in\Z^k$ and $\ve{y} = (y_1,\cdots,y_k) \in\Z^k$. 
Since the coefficients of $x_v$ and $y_v$ in $\psi_{v,I}$ are always $0$, we have $\psi_{v,I\cup\{v\}} = \psi_{v,I\setminus\{v\}}$. For each $\psi\in \Psi$, define $v(\psi)\in\{1,2,\cdots,k\}$ and $I(\psi)\subset\{1,2,\cdots,k\}$ by the condition that $\psi = \psi_{v(\psi),I(\psi)}$. We impose the constraint that $v(\psi)\in I(\psi)$, so that $v(\psi)$ and $I(\psi)$ are uniquely determined by $\psi$. 
Proposition \ref{thm:rep-index} clearly follows from the following two propositions.

\begin{proposition}\label{prop:codim1}
For any linear subvariety $\Pi\subset\{(\ve{x},\ve{y}):\ve{x},\ve{y}\in\R^k\}$ of codimension $1$, the number  of distinct linear forms in $\Psi$ when restricted to $\Pi$ is at least $(k+1)2^{k-2}$. 
\end{proposition}

\begin{proposition}\label{prop:codim2}
For any linear subvariety $\Pi\subset\{(\ve{x},\ve{y}):\ve{x},\ve{y}\in\R^k\}$ of codimension $2$, the number  of distinct linear forms in $\Psi$ when restricted to $\Pi$ is at least $2^{k-1}$. 
\end{proposition}

We will prove them in Sections \ref{sec:Ck-codim1} and \ref{sec:Ck-codim2}, after developing a few preliminary lemmas in Section \ref{sec:Ck-lemmas}. In this section we always use $\psi_1, \psi_2, \cdots$ to denote linear forms in $\Psi$ instead of the ones defined in \eqref{eq:psi-j}.

\subsection{Dependencies among linear forms in $\Psi$}\label{sec:Ck-lemmas}

For a collection $\{\psi_1,\cdots,\psi_s\}\subset \Psi$ of linear forms, denote by $\Pi(\psi_1,\cdots,\psi_s)$ the linear subvariety consisting of those $(\ve{x},\ve{y})$ such that the values $\psi_i(\ve{x},\ve{y})$ are all identical for $1\leq i\leq s$. Generically we expect $\Pi(\psi_1,\cdots,\psi_s)$ to have codimension $s-1$. The following lemmas classify a few non-generic cases.

\begin{lemma}[Non-generic case of three linear forms]\label{lem:3forms}
Let $\psi_1,\psi_2,\psi_3\in \Psi$ be three distinct linear forms. If $\Pi(\psi_1,\psi_2,\psi_3)$ has codimension $1$, then $\psi_1,\psi_2,\psi_3$ share a common set of variables.
\end{lemma}

Here and later, we say that a collection of linear forms $\{\psi_1,\cdots,\psi_s\} \subset \Psi$ shares a common set of variables, if there exists a subset $I\subset\{1,2,\cdots,k\}$ such that $I(\psi_j) = I\cup\{v(\psi_j)\}$ for each $1\leq j\leq s$. In other words, all linear forms $\psi_1,\cdots,\psi_s$ depend only on the variables $\{x_i:i\in I\}$ and $\{y_i:i\notin I\}$.

\begin{proof}
Write $v_j=v(\psi_j)$ and $I_j=I(\psi_j)$ for $j\in\{1,2,3\}$.
Since $\Pi(\psi_1,\psi_2,\psi_3)$ has codimension $1$, there exist nonzero constants $c_1,c_2,c_3\in\R$ with $c_1+c_2+c_3=0$, such that
\[ c_1\psi_1 + c_2\psi_2 + c_3\psi_3 = 0. \]
Examining the coefficients of $x_i$ and $y_i$ in the above equation, we obtain
\begin{equation}\label{eq:3inI} 
c_1(i-v_1)\mathbf{1}_{i\in I_1} + c_2(i-v_2)\mathbf{1}_{i\in I_2} + c_3(i-v_3)\mathbf{1}_{i\in I_3} = 0， 
\end{equation}
and
\begin{equation}\label{eq:3notinI}  
c_1(i-v_1)\mathbf{1}_{i\notin I_1} + c_2(i-v_2)\mathbf{1}_{i\notin I_2} + c_3(i-v_3)\mathbf{1}_{i\notin I_3} = 0，
\end{equation}
for each $1\leq i\leq k$. 
Let $I=I_1\cap I_2\cap I_3$. We show that $I_1 = I \cup \{v_1\}$, and thus similarly $I_2 = I\cup\{v_2\}$ and $I_3 = I\cup\{v_3\}$. To this end, we pick an arbitrary $i_1\in I_1\setminus I$, and prove that $i_1 = v_1$. Since $i_1\notin I$, $i_1$ lies in at most one of $I_2$ and $I_3$. If $i_1$ lies in neither $I_2$ nor $I_3$, then \eqref{eq:3inI} with $i=i_1$ yields
\[ c_1(i_1-v_1) = 0. \]
Since $c_1\neq 0$, we have $i_1=v_1$ as desired.

Now assume that $i_1$ lies in exactly one of $I_2$ and $I_3$. Without loss of generality, assume that $i_1\in I_2$ and $i_1\notin I_3$. Then \eqref{eq:3notinI} with $i=i_1$ yields
\[ c_3(i_1-v_3) = 0. \]
Since $c_3\neq 0$, we have $i_1=v_3$, but this contradicts our restriction that $v_3\in I_3$.
\end{proof}

\begin{lemma}[Non-generic case of five linear forms]\label{lem:5forms}
Let $\psi_1,\cdots,\psi_5\in \Psi$ be five distinct linear forms. If $\Pi(\psi_1,\cdots,\psi_5)$ has codimension at most $2$ , then three of them share a common set of variables.
\end{lemma}

\begin{proof}
Write $v_j=v(\psi_j)$ and $I_j=I(\psi_j)$ for $1\leq j\leq 5$. Suppose, for the purpose of contradiction, that no three of $\psi_1,\cdots,\psi_5$ share a common set of variables. Let $I=I_1\cap\cdots\cap I_5$. We show that $I_1 = I\cup\{v_1\}$, and thus similarly $I_j = I\cup\{v_j\}$ for each $2\leq j\leq 5$. To this end, we pick an arbitrary $i_1\in I_1\setminus I$, and prove that $i_1 = v_1$. We divide into cases according to whether $i_1$ lies in $I_2,\cdots,I_5$ or not.

First assume that $i_1$ lies in none of $I_2,\cdots,I_5$. 
Since $\Pi(\psi_1,\cdots,\psi_5)$ has codimension at most $2$, in particular $\Pi(\psi_1,\cdots,\psi_4)$ has codimension at most $2$. Hence there exist constants $c_1,\cdots,c_4\in\R$, not all zeros, with $c_1+\cdots+c_4=0$, such that
\begin{equation}\label{eq:cpsi4} 
c_1\psi_1 + c_2\psi_2 + c_3\psi_3 + c_4\psi_4 = 0. 
\end{equation}
Examining the coefficients of $x_{i_1}$ in the above equation, we obtain
\begin{equation}\label{eq:i1} 
c_1(i_1-v_1) = 0. 
\end{equation}
If $c_1=0$, then we may apply Lemma \ref{lem:3forms} to $\psi_2,\psi_3,\psi_4$ to conclude that $\psi_2,\psi_3,\psi_4$ share a common set of variables, a contradiction. Hence $c_1\neq 0$, and thus $i_1 = v_1$ as desired.

Next assume that $i_1$ lies in exactly one of $I_2,\cdots,I_5$, say $I_2$. Repeat the argument above with $\psi_1,\psi_3,\psi_4,\psi_5$ (instead of $\psi_1,\psi_2,\psi_3,\psi_4$) to arrive at \eqref{eq:i1} again.

If $i_1$ lies in exactly two of $I_2,\cdots,I_5$, say $I_2$ and $I_3$, then by examining the coefficients of $y_{i_1}$ in \eqref{eq:cpsi4} we obtain
\[ c_4(i_1 - v_4) = 0. \]
If $c_4=0$, then Lemma \ref{lem:3forms} implies that $\psi_1,\psi_2,\psi_3$ share a common set of variables, a contradiction. Hence $c_4\neq 0$, and thus $i_1 = v_4$, but this contradicts our restriction that $v_4\in I_4$.

Finally, if $I_1$ lies in exactly three of $I_2,\cdots,I_5$, say $I_2,I_3,I_4$, then repeat the argument in the previous case with $\psi_1,\psi_2,\psi_3,\psi_5$ to arrive at $i_1 = v_5$, again contradicting our restriction that $v_5 \in I_5$.
\end{proof}

\begin{lemma}[Non-generic case of linear forms restricted to a hyperplane]\label{lem:kforms}
Let $I\subset \{1,2,\cdots,k\}$ be a subset and $\Pi_I$ be a subspace defined by
\[ \Pi_I = \left\{(\ve{x},\ve{y}): \sum_{i\in I}x_i + \sum_{i\notin I}y_i = 0 \right\}. \]
Let $\psi_1,\cdots,\psi_s$ be linear forms in $\Psi$, and let $\widetilde{\psi}_1,\cdots,\widetilde{\psi}_s$ be their restrictions to $\Pi_I$. Suppose that  $\widetilde{\psi}_1,\cdots,\widetilde{\psi}_s$ are all distinct, and that $\Pi(\widetilde{\psi}_1,\cdots,\widetilde{\psi}_s)$ has codimension at most $1$ in $\Pi_I$. Then $s\leq k$.
\end{lemma}

\begin{proof}
Without loss of generality we may assume that $I=\{1,2,\cdots,k\}$, so that $\Pi_I$ is cut out by the equation $x_1+\cdots+x_k=0$.
Write $v_j=v(\psi_j)$ and $I_j=I(\psi_j)$ for $1\leq j\leq s$. It suffices to prove the assertion that each index $i_0$ belongs to either none of $I_j$, or all of $I_j$, or exactly one of $I_j$. Indeed, suppose that this is proved, and let $I_0$ be the intersection $I_1\cap\cdots\cap I_s$. Suppose that $I_1=\cdots=I_t=I_0$ and $I_j\neq I_0$ for $t<j\leq s$. By the assertion, each index not in $I_0$ can appear in at most one of $I_{t+1},\cdots,I_s$. Since each set in $I_{t+1},\cdots,I_s$ contains an index not in $I_0$, we deduce that $s-t\leq k-|I_0|$. Since $\psi_1,\cdots,\psi_t$ are distinct and $I_1=\cdots=I_t$, the values $v_1,\cdots,v_t$ must be distinct, and thus $t\leq |I_0|$. It follows that $s\leq k$ as desired.

To prove the assertion, suppose that $i_0\in I_1$, $i_0\in I_2$, and $i_0\notin I_3$ for some $1\leq i_0\leq k$. Since $\Pi(\widetilde{\psi}_1,\widetilde{\psi}_2,\widetilde{\psi}_3)$ has codimension at most $1$ in $\Pi_I$, there exist nonzero constants $c_1,c_2,c_3\in\R$ with $c_1+c_2+c_3=0$, such that
\[ c_1\widetilde{\psi}_1 + c_2\widetilde{\psi}_2 + c_3\widetilde{\psi}_3 = 0. \]
It follows that
\begin{equation}\label{eq:3psitilde} 
c_1\psi_1 + c_2\psi_2 + c_3\psi_3 = c(x_1+\cdots+x_k) 
\end{equation}
for some $c\in\R$. Examining the coefficients of $y_{i_0}$ in the above equation, we obtain
\[ c_3(i_0-v_3) = 0. \]
Since $c_3\neq 0$, we have $i_0 = v_3$, contradicting the fact that $v_3\in I_3$.
\end{proof}

\subsection{Proof of Proposition \ref{prop:codim1}}\label{sec:Ck-codim1}

Suppose that there is a subspace $\Pi\subset\{(\ve{x},\ve{y}):\ve{x},\ve{y}\in\R^k\}$ of codimension $1$ such that the number of distinct linear forms in $\Psi$ when restricted to $\Pi$ is at most $(k+1)2^{k-2}-1$. Partition $\Psi$ into $m\leq (k+1)2^{k-2}-1$ subsets $\Psi_1,\cdots,\Psi_m$ according to their restrictions to $\Pi$. In other words, the restrictions of $\psi\in\Psi_{j}$ and $\psi'\in\Psi_{j'}$ to $\Pi$ are identical if and only if $j=j'$.

\begin{case}
First suppose that no two forms in the same subset $\Psi_j$ share a common set of variables. Then Lemma \ref{lem:3forms} implies that each $\Psi_j$ contains at most $2$ forms. We may write
\[ \Pi = \{(\ve{x},\ve{y}):\psi_1(\ve{x},\ve{y}) = \psi_2(\ve{x},\ve{y})\} \]
for two distinct forms $\psi_1,\psi_2$ lying in the same $\Psi_j$. We count the number of pairs $(\psi_1',\psi_2')$ with 
\begin{equation}\label{eq:twopairs}
\psi_1'-\psi_2'=c(\psi_1-\psi_2)
\end{equation}
 for some $c\in\R$, and it suffices to show that this number is at most $(k-1)2^{k-2}$. Equivalently, we show that the number of forms not belonging to any pairs is at least $2^{k-1}$. 
We divide into two cases.

If $v(\psi_1)=v(\psi_2)=v$, then the equation $\psi_1=\psi_2$ defining $\Pi$ is of the form
\begin{equation}\label{eq:psi12}  
\sum_{i\in I} \varepsilon_i(i-v)(x_i-y_i) = 0, 
\end{equation}
for some $I\subset [k]\setminus\{v\}$ and $\varepsilon_i\in\{\pm 1\}$. In fact, $I$ is the set of indices lying in exactly one of $I(\psi_1)$ and $I(\psi_2)$.
If \eqref{eq:twopairs} holds, then $v(\psi_1')=v(\psi_2')=v'$, and the equation $\psi_1'=\psi_2'$ is of the form
\begin{equation}\label{eq:psi12'} 
\sum_{i\in I'}\varepsilon_i'(i-v')(x_i-y_i) = 0 
\end{equation}
for some $\varepsilon_i'\in\{\pm 1\}$, where $I'$ is the set of indices lying in exactly one of $I(\psi_1')$ and $I(\psi_2')$. Since \eqref{eq:psi12} and \eqref{eq:psi12'} are the same, we must have $I=I'$ and $v' \notin I$, and moreover either $\varepsilon_i' = \varepsilon_i$ for all $i \in I$ or $\varepsilon_i' = -\varepsilon_i$ for all $i \in I$. Thus for fixed $v' \notin I$, the number of choices for the unordered pair $\{I(\psi_1'),I(\psi_2')\}$ is at most $2^{k-1-|I|}$ (since $v'$ must lie in $I(\psi_1')$ and $I(\psi_2')$). Thus the number of (unordered) pairs $\{\psi_1',\psi_2'\}$ satisfying \eqref{eq:twopairs} is at most 
\[ (k-|I|)2^{k-1-|I|} \leq (k-1)2^{k-2}, \]
as desired.

Now assume that $v(\psi_1)\neq v(\psi_2)$. If \eqref{eq:twopairs} holds, then $v(\psi_1')\neq v(\psi_2')$ as well. In this case if the coefficient of some variable $x_i$ or $y_i$ in $\psi_1$ is nonzero, so is its coefficient in $\psi_1-\psi_2$. The same goes for $\psi_1'-\psi_2'$. Since $\psi_1$ and $\psi_2$ do not involve at least  two variables ($x_{v(\psi_1)}$ or $y_{v(\psi_1)}$ together with $x_{v(\psi_2)}$ or $y_{v(\psi_2)}$), neither $\psi_1'$ nor $\psi_2'$ is allowed to depend on these two variables. There are certainly at least $2^{k-1}$ forms in $\Psi$ involving either of these two variables, and they must appear as singletons in the partition $\Psi_1\cup\cdots\cup\Psi_m$, as desired.
\end{case}

\begin{case}
Now assume that two forms in some subset $\Psi_j$ share a common set of variables. Then $\Pi$ must be of the form
\[ \Pi = \left\{ (\ve{x}, \ve{y}): \sum_{i \in I} x_i + \sum_{i \notin I} y_i = 0 \right\}, \]
for some $I \subset \{1,2,\cdots,k\}$. If two forms $\psi_1, \psi_2$ lie in the same $\Psi_j$, then $\psi_1$ and $\psi_2$ are identical on $\Pi$. Thus they must share the common set of variables $\{x_i : i \in I\}$ and $\{y_i: i \notin I\}$. There are certainly at least $2^{k-1}$ forms in $\Psi$ involving other variables, and they must appear as singletons in the partition $\Psi_1\cup\cdots\cup\Psi_m$, as desired.
\end{case}

\subsection{Proof of Proposition \ref{prop:codim2}}\label{sec:Ck-codim2}
Suppose that there is a subspace $\Pi\subset\{(\ve{x},\ve{y}):\ve{x},\ve{y}\in\R^k\}$ of codimension $2$ such that the number of distinct linear forms in $\Psi$ when restricted to $\Pi$ is at most $2^{k-1}-1$. Partition $\Psi$ into $m\leq 2^{k-1}-1$ subsets $\Psi_1,\cdots,\Psi_m$ according to their restrictions to $\Pi$. In other words, the restrictions of $\psi\in\Psi_{j}$ and $\psi'\in\Psi_{j'}$ to $\Pi$ are identical if and only if $j=j'$.
We divide into two cases.

First suppose that no two forms in the same subset $\Psi_j$ share a common set of variables. Then Lemma \ref{lem:5forms} implies that each $\Psi_j$ contains at most $4$ forms. Since $m\leq 2^{k-1}-1$, this can happen only if $k=3$, in which case $m=3$ and $\Psi_1,\Psi_2,\Psi_3$ all contain exactly $4$ forms. By our assumption, for any choice of $z_i\in\{x_i,y_i\}$ ($i=1,2,3$), the three forms $-z_2-2z_3$, $z_1-z_3$, and $2z_1+z_2$ lie in distinct $\Psi_j$. Since $-z_2-2z_3,z_1-z_3,2z_1+z_2$ form an arithmetic progression with common difference $z_1+z_2+z_3$, it follows that $z_1+z_2+z_3$ restricted to $\Pi$ are identical up to sign for any choice $z_i\in\{x_i,y_i\}$. This contradicts the fact that $\Pi$ has codimension at most $2$.

Now suppose that two forms in some $\Psi_j$ share a common set of variables, so that $\Pi\subset\Pi_I$ for some $I\subset [k]$. For $\psi\in\Psi$ denote by  $\widetilde{\psi}$ its restriction to $\Pi_I$. The number of distinct $\widetilde{\psi}$ as $\psi$ ranges over all forms in $\Psi$ is easily seen to be $k\cdot 2^{k-1}-(k-1)$. By the pigeonhole principle, there must be $k+1$ distinct forms $\widetilde{\psi}_1,\cdots,\widetilde{\psi}_{k+1}$ whose restrictions to $\Pi$ are identical, but this contradicts Lemma \ref{lem:kforms}.

\section{Relative Szemer\'{e}di's theorem for narrow progressions}\label{sec:relative}

To prove the relative Szemer\'{e}di's theorem for narrow progressions (Theorem \ref{thm:relative-szemeredi}), it suffices to prove the following transference principle.

\begin{theorem}[Transference]\label{thm:transference}
Let $k\geq 2$ be a positive integer. Let $N$ be a sufficiently large prime, and let $G = \Z/N\Z$. Let $f,\nu: G\rightarrow\R$ be functions satisfying $0 \leq f\leq \nu$. Let $D, S \geq 2$ be positive integers satisfying $S = o(D)$. Suppose that $\nu$ satisfies the $k$-linear forms conditions with width $S$. Then  there exists a function $\widetilde{f}:G\rightarrow [0,1]$ with $\E f= \E \widetilde{f}+o(1)$, such that
\[ |\Lambda_D(f,\cdots,f)-\Lambda_D(\widetilde{f},\cdots,\widetilde{f})|=o(1). \]
\end{theorem}

\begin{proof}[Proof of Theorem \ref{thm:relative-szemeredi} assuming Theorem \ref{thm:transference}]
Apply Theorem \ref{thm:transference} to obtain the bounded function $\widetilde{f}$. Since $\E f\geq\delta$ we have $\E\widetilde{f} \geq \delta/2$, and it suffices to show that $\Lambda_D(\widetilde{f},\cdots,\widetilde{f}) \gg_{\delta} 1$. For each $m\in G$, let $\widetilde{f}_m: [D]\rightarrow [0,1]$ be the function defined by $\widetilde{f}_m (n) = \widetilde{f}(m+n)$. Let $M \subset G$ be the set of $m \in G$ with $\E \widetilde{f}_m \geq \delta/4$. From the inequalities
\[ \frac{\delta}{2} \leq \E \widetilde{f} = \E_{m\in G} \E \widetilde{f}_m \leq \frac{\delta}{4} + \frac{|M|}{|G|}, \]
we conclude that $|M| \geq \delta |G|/4$. For each $m\in G$ we apply (the quantitative version of) Szemer\'{e}di's theorem (see for example \cite[Proposition 2.3]{GT08}) after embedding $[D]$ into a cyclic group to obtain
\[ \E_{n,d\in [D]} \widetilde{f}_m(n) \widetilde{f}_m(n+d) \cdots \widetilde{f}_m(n+(k-1)d) \gg_{k,\delta} 1. \]
Here we naturally set $\widetilde{f}_m(n) = 0$ for $n\notin [D]$. Averaging this over all $m\in G$, we arrive at
\[ \E_{m\in G} \E_{n,d\in [D]} \widetilde{f}(m+n) \widetilde{f}(m+n+d) \cdots \widetilde{f}(m+n+(k-1)d) \gg_{k,\delta} 1. \]
This is equivalent to the desired claim $\Lambda_D(\widetilde{f},\cdots,\widetilde{f}) \gg_{\delta} 1$ after a change of variables.
\end{proof}

The proof of Theorem \ref{thm:transference}, motivated by arguments in \cite{CFZ13,Zhao14}, is split into two parts. In the first part, we find a bounded model $\widetilde{f}: G\rightarrow [0,1]$ for $f$ in the sense that $\|f-\widetilde{f}\|_D$ is small, where the norm $\| \cdot \|_D$ is defined as follows.

\begin{definition}
Fix a positive integer $k \geq 2$. For any function $f: G\rightarrow \R$ and any $1 \leq i \leq k$, define
\[ \| f \|_{D,i} = \sup \left| \Lambda_D(f_1,\cdots,f_{i-1},f,f_{i+1},\cdots,f_k) \right|, \]
where the supremum is taken over all functions $f_1,\cdots,f_{i-1},f_{i+1},\cdots,f_k: G\rightarrow [-1,1]$. Furthermore, define
\[ \| f\|_D = \sup_{1 \leq i \leq k} \|f\|_{D,i}. \]
\end{definition}

It can be easily verified that these are indeed norms; however, we will not need this fact. 

\begin{proposition}[Approximation by bounded functions]\label{prop:approx-by-bounded}
Let $k\geq 2$ be a positive integer. Let $N$ be a sufficiently large prime, and let $G = \Z/N\Z$. Let $f,\nu: G\rightarrow\R$ be functions satisfying $0 \leq f\leq \nu$. Let $D, S \geq 2$ be positive integers satisfying $S = o(D)$. Suppose that $\nu$ satisfies the $k$-linear forms conditions with width $S$. Then there exists a function $\widetilde{f}:G\rightarrow [0,1]$ with $\E f= \E \widetilde{f} + o(1)$, such that $\| f - \widetilde{f} \|_D = o(1)$.
\end{proposition}

In the second part of the proof of Theorem \ref{thm:transference}, we show that the $k$-AP counts for the original function $f$ and for its bounded model $\widetilde{f}$ are close.

\begin{proposition}[Counting lemma]\label{prop:counting}
Let $k\geq 2$ be a positive integer. Let $N$ be a sufficiently large prime, and let $G = \Z/N\Z$. Let $D, S \geq 2$ be positive integers satisfying $S = o(D)$. Let $\nu: G\rightarrow\R$ be a function satisfying the $k$-linear forms conditions with width $S$. For $1\leq i\leq k$, let $f_i,\widetilde{f}_i,\nu_i:G\rightarrow \R$ be functions with $\nu_i\in\{\nu,1\}$, $0 \leq f_i \leq \nu_i$ and $ 0 \leq \widetilde{f}_i \leq 1$. If $\| f_i - \widetilde{f}_i \|_{D,i} = o(1)$ for each $1 \leq i \leq k$, then 
\[ \left| \Lambda_D(f_1,\cdots,f_k)-\Lambda_D(\widetilde{f}_1,\cdots,\widetilde{f}_k) \right| = o(1). \]
\end{proposition}

Clearly Theorem \ref{thm:transference} follows by combining Propositions \ref{prop:approx-by-bounded} and \ref{prop:counting}. The proof of Proposition \ref{prop:approx-by-bounded}, presented in Section \ref{sec:dense-model}, follows closely the proof of \cite[Lemma 3.3]{Zhao14}, using the Green-Tao-Ziegler dense model theorem. The proof of Proposition \ref{prop:counting}, presented in Section \ref{sec:counting-lemma}, follows closely a densification argument in \cite[Section 6]{CFZ13}.

\section{The dense model theorem}\label{sec:dense-model}

In this section we prove Proposition \ref{prop:approx-by-bounded}. The main tool used is the Green-Tao-Ziegler dense model theorem. Indeed, for each $1 \leq i \leq k$ a straightforward application of this dense model theorem produces a function $\widetilde{f}_i$ such that $\| f - \widetilde{f}_i \|_{D,i} = o(1)$. However, some extra efforts are needed to obtain a single model $\widetilde{f}$ that is close to $f$ in the norm $\| \cdot \|_{D,i}$ for every $i$. To achieve this, we define the following stronger notion of closeness (compare with \cite[Definition 3.1]{Zhao14}).

\subsection{Discrepancy pairs}

\begin{definition}[Discrepancy pair]\label{def:disc-pair}
Fix a positive integer $\ell$ and a linear form $\xi: \Z^{\ell}\rightarrow\Z$ in $\ell$ variables. Let $S \geq 2$ be a positive integer and $\varepsilon > 0$ be real. For two functions $f,\widetilde{f}: G\rightarrow\R$, we say that $(f,\widetilde{f})$ is an $\varepsilon$-discrepancy pair with width $S$ with respect to $\xi$, if for all functions $u_1, \cdots, u_{\ell} : G^{\ell+1} \rightarrow [-1,1]$ with $u_i$ not depending on the $(i+1)$th coordinate, we have
\begin{equation}\label{eq:disc-pair} 
\left| \E_{n \in G} \E_{\ve{s} \in [S]^{\ell}} \left( f(n+\xi(\ve{s})) - \widetilde{f}(n+\xi(\ve{s})) \right) \prod_{i=1}^{\ell} u_i(n,\ve{s}) \right| \leq \varepsilon.
\end{equation}
\end{definition}

Note that, if $\ve{s} = (s_1,\cdots,s_{\ell})$, then the value of $u_i(n, \ve{s})$ does not depend on $s_i$. We will be interested in discrepancy pairs with respect to $\psi$ and $\psi_j$ defined in \eqref{eq:psi-j} and \eqref{eq:psi}. Note that $\psi$ is a linear form in $k$ variables, while each $\psi_j$ is a linear form in $k-1$ variables. The following two lemmas imply that discrepancy pairs with respect to $\psi$ are automatically discrepancy pairs with respect to every $\psi_j$.

\begin{lemma}\label{lem:disc-pair-1}
Let $\xi: \Z^{\ell} \rightarrow \Z$ be a linear form in $\ell$ variables, and let $\xi': \Z^{\ell-1} \rightarrow \Z$ be the linear form defined by $\xi'(\ve{s}) = \xi(\ve{s}, 0)$ for any $\ve{s} \in \Z^{\ell-1}$. Let $S \geq 2$ be a positive integer and $\varepsilon > 0$ be real. If $(f,\widetilde{f})$ is an $\varepsilon$-discrepancy pair with width $S$ with respect to $\xi$, then $(f,\widetilde{f})$ is also an $\varepsilon$-discrepancy pair with width $S$ with respect to $\xi'$.
\end{lemma}

\begin{proof}
Let $u_1', \cdots, u_{\ell-1}': G^{\ell} \rightarrow [-1,1]$ be arbitrary functions with $u_i'$ not depending on the $(i+1)$th coordinate. By definition, it suffices to show that
\[ \left| \E_{n \in G} \E_{\ve{s} \in [S]^{\ell-1}} \left( f(n+\xi'(\ve{s})) - \widetilde{f}(n+\xi'(\ve{s})) \right) \prod_{i=1}^{\ell-1} u_i'(n,\ve{s}) \right| \leq \varepsilon. \]
Introducing a new variable $s_{\ell} \in [S]$, and note that $\xi(\ve{s}, s_{\ell}) = \xi'(\ve{s}) + as_{\ell}$ for some $a \in \Z$. After translating $n$ by $as_{\ell}$ and averaging over $s_{\ell}$, we may rewrite the average above as
\[ \E_{n \in G} \E_{\ve{s} \in [S]^{\ell-1}} \E_{s_{\ell} \in [S]} \left( f(n+\xi(\ve{s}, s_{\ell})) - \widetilde{f}(n+\xi(\ve{s}, s_{\ell})) \right) \prod_{i=1}^{\ell-1} u_i'(n + as_{\ell},\ve{s}). \]
This can be further rewritten in the form
\[ \E_{n \in G} \E_{\ve{s} = (s_1, \cdots, s_{\ell}) \in [S]^{\ell}} \left( f(n + \xi(\ve{s})) - \widetilde{f}(n + \xi(\ve{s})) \right) \prod_{i=1}^{\ell} u_i(n, \ve{s}), \]
where $u_i$ is defined by $u_i(n, s_1,\cdots,s_{\ell}) = u_i'(n + as_{\ell}, s_1,\cdots,s_{\ell-1})$ for $i \leq \ell-1$, and $u_{\ell}(n, s_1,\cdots,s_{\ell}) = 1$, so that $u_i(n, s_1,\cdots,s_{\ell})$ does not depend on $s_i$ for every $i$. Thus the average above is indeed bounded by $\varepsilon$ since $(f,\widetilde{f})$ is an $\varepsilon$-discrepancy pair with respect to $\xi$.
\end{proof}

\begin{lemma}\label{lem:disc-pair-2}
Let $\xi, \xi': \Z^{\ell} \rightarrow \Z$ be two linear forms in $\ell$ variables defined by
\[ \xi(s_1, \cdots, s_{\ell}) = a_1s_1 + \cdots + a_{\ell}s_{\ell}, \ \ \xi'(s_1, \cdots, s_{\ell}) = a_1's_1 + \cdots + a_{\ell}'s_{\ell}, \]
for some $a_1,\cdots,a_{\ell}, a_1',\cdots,a_{\ell}' \in \Z\setminus\{0\}$ such that $a_i'$ divides $a_i$ and $a_i/a_i'$ divides $Q$ for each $i$, where $Q$ is a positive integer.  Let $S \geq 2$ be a positive integer and $\varepsilon > 0$ be real. If $(f,\widetilde{f})$ is an $\varepsilon$-discrepancy pair with width $S$ with respect to $\xi$, then $(f,\widetilde{f})$ is also an $\varepsilon$-discrepancy pair with width $QS$ with respect to $\xi'$.
\end{lemma}

\begin{proof}
Let $u_1', \cdots, u_{\ell}' : G^{\ell+1} \rightarrow [-1,1]$ be arbitrary functions with $u_i'$ not depending on the $(i+1)$-th coordinate. Write $S' = QS$. By definition, it suffices to show that
\[ \left| \E_{n \in G} \E_{\ve{s} \in [S']^{\ell}} \left( f(n+\xi'(\ve{s})) - \widetilde{f}(n+\xi'(\ve{s})) \right) \prod_{i=1}^{\ell} u_i'(n,\ve{s}) \right| \leq \varepsilon. \]
For each $i$ let $q_i = |a_i/a_i'|$. Since $q_i$ divides $Q$, we may split the interval $[S']$ into $Q$ arithmetic progressions $P_{i1}, \cdots, P_{iQ}$, each of which has length $S$ and common difference $q_i$. It suffices to show that, for any choice of $P_i \in \{P_{i1}, \cdots, P_{iQ}\}$ we always have
\[ \left| \E_{n \in G} \E_{\ve{s} \in P_1 \times \cdots \times P_{\ell}} \left( f(n+\xi'(\ve{s})) - \widetilde{f}(n+\xi'(\ve{s})) \right) \prod_{i=1}^{\ell} u_i'(n,\ve{s}) \right| \leq \varepsilon. \]
Write $\ve{s} = (s_1, \cdots, s_{\ell})$. For $s_i \in P_i$, make the change of variable $s_i = q_i t_i + r_i$ with $t_i \in [S]$. Since $\xi'(\ve{s}) = \xi(\ve{t}) + r$ where $\ve{t} = (t_1,\cdots,t_{\ell})$ and $r = a_1'r_1 + \cdots + a_{\ell}'r_{\ell}$, the inequality above is equivalent to
\[ \left| \E_{n \in G} \E_{\ve{t} \in [S]^{\ell}} \left( f(n+\xi(\ve{t})) - \widetilde{f}(n+\xi(\ve{t})) \right) \prod_{i=1}^{\ell} u_i(n,\ve{t}) \right| \leq \varepsilon, \]
where $u_i(n,t_1,\cdots,t_{\ell}) = u_i'(n-r, q_1t_1+r_1, \cdots, q_{\ell}t_{\ell}+r_{\ell})$. This follows from the assumption that $(f, \widetilde{f})$ is an $\varepsilon$-discrepancy pair with width $S$ with respect to $\xi$, since $u_i$ does not depend on $t_i$.
\end{proof}

\begin{lemma}\label{lem:disc-pair}
Let $S, D\geq 2$ be positive integers with $S = o(D)$. If $(f,\widetilde{f})$ is an $o(1)$-discrepancy pair with width $S$ with respect to $\psi_i$ for some $1 \leq i\leq k$, and moreover $\|f - \widetilde{f} \|_{L^1} = O(1)$, then $\|f - \widetilde{f}\|_{D,i} = o(1)$.
\end{lemma}

\begin{proof}
Without loss of generality, we may assume that $i=1$, so that $\psi_1$ is a linear form in the $k-1$ variables $s_2,\cdots,s_k$ defined in \eqref{eq:psi-j}. Let $f_2,\cdots,f_k: G \rightarrow [-1,1]$ be arbitrary functions. Fix an arbitrary $s_1 \in \Z$. For $2 \leq i \leq k$, define the function $u_i: G^k \rightarrow [-1,1]$ by
\[ u_i(n,s_2\cdots,s_k) = f_i(n+\psi_i(s_1,\cdots,s_k)). \]
Note that $u_i$ does not depend on $s_i$. Since $(f,\widetilde{f})$ is an $o(1)$-discrepancy pair with width $S$ with respect to $\psi_1$, we have
\begin{equation}\label{eq:disc-pair-lemma} 
\left| \E_{n \in G} \E_{s_2,\cdots,s_k \in [S]} \left( f(n+\psi_1(\ve{s})) - \widetilde{f}(n+\psi_1(\ve{s})) \right) \prod_{i=2}^k f_i(n + \psi_i(\ve{s})) \right| = o(1),
\end{equation}
for every $s_1 \in \Z$, where $\ve{s} = (s_1,\cdots,s_k)$. Note that
\[ \Lambda_D(f-\widetilde{f}, f_2,\cdots,f_k) = \E_{n \in G} \E_{d \in [D]} \left( f(n) - \widetilde{f}(n) \right)  \prod_{i=2}^k f_i(n + (i-1)d). \]
Introduce new variables $s_2,\cdots,s_k$ taking values in $[S]$. Shifting $d$ by $s_2+\cdots+s_k$ causes an error bounded by
\[ O \left( \frac{S}{D} \E_{n \in G} \left| f(n) - \widetilde{f}(n) \right| \right) = o\left( \|f-\widetilde{f}\|_{L^1} \right) = o(1) \]
by hypothesis. Thus
\[ \Lambda_D(f-\widetilde{f}, f_2,\cdots,f_k) =   \E_{d \in [D]} \E_{n \in G} \E_{s_2,\cdots,s_k \in [S]} \left( f(n) - \widetilde{f}(n) \right)  \prod_{i=2}^k f_i(n + (i-1)(d+s_2+\cdots+s_k)) + o(1). \]
After renaming $d$ by $s_1$ and replacing $n$ by $n + \psi_1(\ve{s})$, we may transform this into 
\[ \E_{s_1 \in [D]} \E_{n \in G} \E_{s_2,\cdots,s_k \in [S]} \left( f(n + \psi_1(\ve{s})) - \widetilde{f}(n + \psi_1(\ve{s})) \right) \prod_{i=2}^k f_i(n + \psi_i(\ve{s}) ) + o(1). \]
By \eqref{eq:disc-pair-lemma}, for each $s_1$ the inner average above is $o(1)$. This completes the proof.
\end{proof}

\subsection{Proof of Proposition~\ref{prop:approx-by-bounded}}

For a positive integer $S \geq 2$, let $\mathcal{F}_S$ be the collection of all functions that are convex combinations of functions $u:G\rightarrow\R$ of the form
\begin{equation}\label{eq:u} 
u(n)=\E_{\ve{s} \in [S]^k} \prod_{i=1}^k u_i(n - \psi(\ve{s}), \ve{s})
\end{equation}
for some $u_1,\cdots,u_k:G^{k+1} \rightarrow [-1,1]$ with $u_i(n,\ve{s})$ not depending on $s_i$, where $\psi$ is defined in \eqref{eq:psi}. In view of Lemma \ref{lem:disc-pair}, to prove Proposition \ref{prop:approx-by-bounded} it suffices to find $\widetilde{f}$ such that $(f,\widetilde{f})$ forms an $o(1)$-discrepancy pair with width $o(D)$ with respect to each $\psi_i$. By Lemmas~\ref{lem:disc-pair-1} and~\ref{lem:disc-pair-2}, it suffices to find $\widetilde{f}$ such that $(f,\widetilde{f})$ forms an $o(1)$-discrepancy pair with width $o(D)$ with respect to $\psi$. In other words, we need to ensure that
\[ \left| \langle f - \widetilde{f}, u \rangle \right| = o(1) \]
for any $u \in \mathcal{F}_{o(D)}$. This will be achieved by the Green-Tao-Ziegler dense model theorem \cite{GT08,TZ08} (with simplified proofs in \cite{Gow10, RTTV08}).

\begin{lemma}[Green-Tao-Ziegler dense model theorem]\label{lem:green-tao-ziegler}
For any $\varepsilon > 0$, there is a positive integer $K = K(\varepsilon)$ and a positive constant $\varepsilon' = \varepsilon'(\varepsilon)$ such that the following statement holds. Let $\mathcal{F}$ be an arbitrary collection of functions $u: X \rightarrow [-1,1]$ on a finite set $X$. Let $\nu: X \rightarrow \R_{\geq 0}$ be a function satisfying
\[ \left| \langle \nu-1, u \rangle \right| \leq \varepsilon' \]
for all $u \in \mathcal{F}^K$, where $\mathcal{F}^K$ consists of all functions of the form $u_1u_2 \cdots u_K$ with each $u_i \in \mathcal{F}$. For any function $f: X \rightarrow \R_{\geq 0}$ with $f \leq \nu$ and $\E f \leq 1$, there is a function $\widetilde{f}: X \rightarrow [0,1]$ with the properties that $\E \widetilde{f} = \E f$, and moreover
\[ \left| \langle f - \widetilde{f}, u \rangle \right| \leq \varepsilon \]
for all $u \in \mathcal{F}$.
\end{lemma}

In view of this, the task of proving Proposition~\ref{prop:approx-by-bounded} reduces to proving the following two lemmas, the first of which saying that $\mathcal{F}_S$ is almost closed under pointwise multiplication, and the second of which verifies the hypothesis in Lemma~\ref{lem:green-tao-ziegler} about the majorant $\nu$.

\begin{lemma}\label{lem:FS-stable}
Let $K$ be a positive integer and $\varepsilon \in (0,1)$ be real. Let $S, T \geq 2$ be positive integers with $T \leq \varepsilon S$. For any function $u \in \mathcal{F}_S^K$,  there is a function $v\in \mathcal{F}_{T}$ satisfying $\|v-u\|_{\infty} = O(K \varepsilon)$. 
\end{lemma}

\begin{proof}
It suffices to prove this when $u = u_1u_2 \cdots u_K$ and each $u_j: G\rightarrow [-1,1]$ is of the form
\[ u_j(n)=\E_{\ve{s} \in [S]^k} \prod_{i=1}^k u_{ji}(n - \psi(\ve{s}), \ve{s}), \]
for some $u_{ji} : G^{k+1} \rightarrow [-1,1]$ with $u_{ji}(n, \ve{s})$ not depending on $s_i$, since every function in $\mathcal{F}_S^K$ is a convex combination of these functions $u$. We may write
\[ u(n) = \E_{\ve{s}_1, \cdots, \ve{s}_K \in [S]^k} \prod_{j=1}^K \prod_{i=1}^k u_{ji}(n - \psi(\ve{s}_j), \ve{s}_j). \]
Introduce the auxiliary variables $\ve{t} = (t_1,\cdots,t_k) \in [T]^k$, and note that translating each $s_{ji}$ ($1 \leq j \leq K$) by $t_i$ changes the average by $O(K \varepsilon)$. Thus
\[ u(n) = \E_{\ve{s}_1, \cdots, \ve{s}_K \in [S]^k}  \E_{\ve{t} \in [T]^k} \prod_{i=1}^k \prod_{j=1}^K u_{ji}(n - \psi(\ve{s}_j) - \psi(\ve{t}), \ve{s}_j + \ve{t}) + O(K \varepsilon). \]
For fixed $\ve{s}_1, \cdots, \ve{s}_K \in [S]^k$, consider the function $v: G\rightarrow [-1,1]$ defined by
\[ v(n) = \E_{\ve{t} \in [T]^k} \prod_{i=1}^k v_i(n - \psi(\ve{t}), \ve{t}), \]
where $v_i: G^{k+1} \rightarrow [-1,1]$ is defined by
\[ v_i(n, \ve{t}) = \prod_{j=1}^K u_{ji}(n - \psi(\ve{s}_j), \ve{s}_j + \ve{t}). \]
Thus we have approximated $u$ by a convex combination of these functions $v$, up to an error of $O(K\varepsilon)$ in the $L^{\infty}$-norm. Since $v_i(n, \ve{t})$ does not depend on $t_i$, we have $v \in \mathcal{F}_T$. This completes the proof.
\end{proof}

\begin{lemma}\label{lem:nu-1}
Let $S \geq 2$ be a positive integer. If $\nu$ satisfies the $k$-linear forms conditions with width $S$, then 
\[ \left| \langle \nu-1, u \rangle \right| = o(1) \]
for any $u \in \CF_S$.
\end{lemma}

\begin{proof}
It suffices to prove this for $u$ of the form \eqref{eq:u}. We may write
\[ \langle \nu-1, u \rangle = \E_{n \in G} \E_{\ve{s} \in [S]^k} \left( \nu(n+\psi(\ve{s}))-1 \right) \prod_{i=1}^k u_i(n, \ve{s}) \]
after a change of variable, where $u_i(n, \ve{s})$ does not depend on $s_i$. To upper bound this, we will apply Cauchy-Schwarz inequality $k$ times, with respect to the variables $s_i$ in the $i$th step. Since $u_i$ does not depend on $s_i$, the Cauchy-Schwarz step with respect to $s_i$ eliminates the function $u_i$. In the end we arrive at
\[ \left| \langle \nu-1, u \rangle \right|^{2^{k}} \leq \E_{n \in G} \E_{\ve{s}^{(0)},\ve{s}^{(1)} \in [S]^k} \prod_{\omega \in \{0,1\}^k} \left( \nu(n+\psi(\ve{s}^{(\omega)})) - 1 \right). \]
This is $o(1)$ after expanding out the product since $\nu$ satisfies (the second set of) the $k$-linear forms conditions with width $S$.
\end{proof}

\begin{proof}[Proof of Proposition~\ref{prop:approx-by-bounded}]
Let $\varepsilon = \varepsilon(N) > 0$ be a function decaying to zero sufficiently slowly. Since $\nu$ satisfies the $k$-linear forms conditions with width $S$, we have by Lemma~\ref{lem:nu-1}
\[ \left| \langle \nu-1, v \rangle \right| = o(1) \]
for any $v \in \mathcal{F}_S$. Choose a positive integer $S'$ such that $S' = o(D)$ and $S = o(S')$. Let $K = K(\varepsilon)$ and $\varepsilon' = \varepsilon'(\varepsilon)$ be constants from the Green-Tao-Ziegler dense model theorem. For any function $u \in \mathcal{F}_{S'}^K$, we have by Lemma~\ref{lem:FS-stable} an approximation $v \in \mathcal{F}_S$ satisfying 
\[ \| v - u \|_{\infty} = o(K) \leq \varepsilon'/4, \]
provided that $\varepsilon$ decays slowly enough compared to the decay rate in $S = o(S')$. Thus for any $u \in \mathcal{F}_{S'}^K$ we have
\[ | \langle \nu-1, u \rangle | \leq o(1) + \frac{1}{4} \varepsilon' \|\nu-1\|_{L^1} \leq o(1) + \frac{1}{2} \varepsilon' \leq \varepsilon'.  \]
By the Green-Tao-Ziegler dense model theorem, we may find $\widetilde{f}: G\rightarrow [0,1]$ with the properties that $\E \widetilde{f} = \E f$, and moreover
\[ | \langle f-\widetilde{f}, u\rangle | \leq \varepsilon \]
for any $u \in \mathcal{F}_{S'}$. By definition, this implies that $(f,\widetilde{f})$ is an $\varepsilon$-discrepancy pair with width $S'$ with respect to $\psi$. It follows from Lemmas~\ref{lem:disc-pair-1} and~\ref{lem:disc-pair-2} that $(f,\widetilde{f})$ is an $\varepsilon$-discrepancy pair with width $S$ with respect to each $\psi_i$. Since
\[ \|f - \widetilde{f}\|_{L^1} \leq \|f\|_{L^1} + 1 \leq \|\nu\|_{L^1} + 1 = 2+o(1), \]
we may apply Lemma~\ref{lem:disc-pair} to conclude that
\[ \|f - \widetilde{f}\|_{D,i} = o(1) \]
for each $1 \leq i \leq k$. This completes the proof.
\end{proof}

\section{The counting lemma}\label{sec:counting-lemma}

In this section we prove Proposition \ref{prop:counting} by induction on the number of indices $i$ with $\nu_i\neq 1$. Consider first the base case when $\nu_i=1$ for all $i$. Note that
\[ \Lambda_D(f_1,\cdots,f_k)-\Lambda_D(\widetilde{f}_1,\cdots,\widetilde{f}_k) = \sum_{i=1}^k \Lambda_D(\widetilde{f}_1,\cdots,\widetilde{f}_{i-1},f_i-\widetilde{f}_i,f_{i+1},\cdots,f_k). \]
For each $1 \leq i \leq k$, since $\widetilde{f}_1,\cdots,\widetilde{f}_{i-1},f_{i+1},\cdots,f_k$ are all bounded by $1$, the $i$th summand is bounded in absolute value by $\|f_i - \widetilde{f}_i\|_{D,i}$. The conclusion follows immediately.

We now turn to the inductive step. Assume that $\nu_j \neq 1$ for some $1 \leq j \leq k$, and without loss of generality we may assume that $\nu_1 \neq 1$. We split the difference $\Lambda_D(f_1,\cdots,f_k)-\Lambda_D(\widetilde{f}_1,\cdots,\widetilde{f}_k)$ into the sum of
\[ \Lambda_D(f_1 - \widetilde{f}_1,\widetilde{f}_{2},\cdots,\widetilde{f}_k) \]
and
\[ \E_{n\in G} f_1(n) \E_{d\in [D]} \left(\prod_{i=2}^k f_i(n+(i-1)d) - \prod_{i=2}^k \widetilde{f}_i(n+(i-1)d) \right). \]
The first expression is bounded in absolute value by $\|f_1 - \widetilde{f}_1\|_{D,1} = o(1)$ since all $\widetilde{f}_i$ are bounded by $1$. Thus it suffices to show that the second expression is $o(1)$. 

To simplify the notations, define $f_1',\widetilde{f}_1':G\rightarrow\R$ by
\[ f_1'(n)=\E_{d\in [D]} \prod_{i=2}^{k} f_i(n+(i-1)d),\ \ \widetilde{f}_1'(n) = \E_{d\in [D]} \prod_{i=2}^{k} \widetilde{f}_i(n+(i-1)d), \]
and define also $\nu_1':G\rightarrow\R$ similarly by
\[ \nu_1'(n)=\E_{d\in [D]} \prod_{i=2}^{k} \nu_i(n+(i-1)d). \]
Clearly $0 \leq f_1' \leq \nu_1'$ and $0 \leq \widetilde{f}_1' \leq 1$, and our goal is to show that
\[ \E_{n \in G} f_1(n) (f_1'(n) - \widetilde{f}_1'(n)) = o(1). \]
After an application of Cauchy-Schwarz and using $0 \leq f_1 \leq \nu$, the task becomes to show that
\[ \left(\E_{n\in G}\nu(n)\right) \E_{n\in G} \nu(n) (f_1'(n) - \widetilde{f}_1'(n)) ^2 = o(1). \]
Since the average of $\nu$ is $1+o(1)$, it suffices to prove the inequalities
\begin{equation}\label{eq:counting1} 
\E_{n\in G} (\nu(n)-1) (f_1'(n) - \widetilde{f}_1'(n)) ^2  = o(1) 
\end{equation}
and
\begin{equation}\label{eq:counting2} 
\E_{n\in G}  (f_1'(n) - \widetilde{f}_1'(n)) ^2 = o(1). 
\end{equation}

\subsection{Proof of \eqref{eq:counting1}}

Expanding the square and recalling the definitions of $f_1'$ and $\widetilde{f}_1'$, we get four terms of the form
\begin{equation}\label{eq:counting11}  
\E_{n\in G} (\nu(n)-1) \E_{d^{(0)},d^{(1)}\in [D]} \prod_{i=2}^k \prod_{\tau\in\{0,1\}} f_i^{(\tau)}(n+(i-1)d^{(\tau)}) , 
\end{equation}
where $f_i^{(\tau)}\in\{f_i,\widetilde{f}_i\}$. It suffices to show that each term is $o(1)$. Introduce new variables $s_i \in [S]$ for each $i>1$, and translate both $d^{(0)}$ and $d^{(1)}$ by $s_2 + \cdots + s_k$. This causes an error bounded by
\begin{equation}\label{eq:error} 
O\left( D^{-2} \E_{n \in G} (\nu(n)+1) \sum_{(d^{(0)}, d^{(1)}) \in \Omega\cap\Z^2} \prod_{i=2}^k \prod_{\tau\in\{0,1\}} f_i^{(\tau)}(n+(i-1)d^{(\tau)}) \right),
\end{equation}
where $\Omega \subset \R^2$ is the region defined by
\[ \Omega = [1,D+kS]^2 \setminus [kS, D-kS]^2. \]
Note that the area of $\Omega$ is $\asymp DS$. This error is $o(1)$ since $\nu$ satisfies (the third set of) the $k$-linear forms conditions with width $S$. Thus \eqref{eq:counting11} is
\[ \E_{n \in G} (\nu(n)-1) \E_{\ve{s} = (s_2,\cdots,s_k) \in [S]^{k-1}} \E_{d^{(0)},d^{(1)} \in [D]} \prod_{i=2}^k \prod_{\tau\in\{0,1\}} f_i^{(\tau)}(n + (i-1)(d^{(\tau)} + s_2 + \cdots + s_k ) + o(1). \]
Now replace $n$ by $n+\psi_1(s_2,\cdots,s_k)$ to see that \eqref{eq:counting11} is
\[  \E_{n\in G}  \E_{\ve{s} = (s_2,\cdots,s_k) \in [S]^{k-1}} (\nu(n+\psi_1(\ve{s})-1) \E_{d^{(0)},d^{(1)}\in [D]} \prod_{i=2}^k \prod_{\tau\in\{0,1\}} f_i^{(\tau)}(n+\psi_i(d^{(\tau)}, \ve{s})) + o(1). \]
This is $o(1)$ using the $k$-linear forms conditions on $\nu$ after applying Cauchy-Schwarz inequality $k-1$ times with respect to $s_2, \cdots, s_k$. See the following lemma for details.


\begin{lemma}[Gowers-Cauchy-Schwarz]\label{lem:gcs}
Let $S \geq 2$ be real. Let $\nu:G\rightarrow\R_{\geq 0}$ be a function satisfying the $k$-linear forms conditions with width $S$. For $2\leq i\leq k$ and $\tau\in\{0,1\}$, let $\nu_i^{(\tau)}$ be either $\nu$ or $1$ and let $f_i^{(\tau)}:G\rightarrow\R_{\geq 0}$ be a function with $f_i^{(\tau)} \leq \nu_i^{(\tau)}$. Let $S_1,\cdots,S_k \geq S$ be positive integers. For each $1\leq \ell \leq k$, define
\begin{align*} 
I_{\ell} = \E_{n\in G} \E_{\substack{s_1^{(0)},\cdots,s_{\ell}^{(0)}\\ s_1^{(1)},\cdots,s_{\ell}^{(1)}}} \E_{s_{\ell+1},\cdots,s_k} & \prod_{\omega\in\{0,1\}^{[\ell]\setminus\{1\}}} (\nu(n+\psi_1(\ve{s}^{(\omega)})) - 1) \\
& \prod_{i=2}^{\ell} \prod_{\omega\in\{0,1\}^{[\ell]\setminus\{i\}}} \nu_i^{(\omega_1)}(n+\psi_i(\ve{s}^{(\omega)})) \prod_{i=\ell+1}^k \prod_{\omega\in\{0,1\}^{[\ell]}} f_i^{(\omega_1)}(n+\psi_i(\ve{s}^{(\omega)})),
\end{align*}
where the average over $s_i^{(0)},s_i^{(1)}$ (or $s_i$) is understood to be in the range $[S_i]$ ($1\leq i\leq k$).
Then $I_{\ell}=o(1)$ for each $1\leq \ell\leq k$.
\end{lemma}

Here we adopted the natural convention that $\ve{s}^{(\omega)} = (s_1^{(\omega_1)},\cdots,s_{\ell}^{(\omega_{\ell})},s_{\ell+1},\cdots,s_k)$ for $\omega \in \{0,1\}^{[\ell]}$.

\begin{proof}
First note that $I_k = o(1)$ follows from the $k$-linear forms conditions on $\nu$. Thus it suffices to show that $I_{\ell-1}^2 \leq (1+o(1)) I_{\ell}$ for any $2\leq \ell \leq k$. After pulling out the terms involving $f_{\ell}$ which do not depend on the variable $s_{\ell}$, we can rewrite $I_{\ell-1}$ as
\begin{align*} 
 \E_{n\in G} & \E_{\substack{s_1^{(0)},\cdots,s_{\ell-1}^{(0)} \\ s_1^{(1)},\cdots,s_{\ell-1}^{(1)}}} \E_{s_{\ell+1},\cdots,s_k} \prod_{\omega \in\{0,1\}^{[\ell-1]}} f_{\ell}^{(\omega_1)}(n+\psi_{\ell}(\ve{s}^{(\omega)}))  \E_{s_{\ell}} \prod_{\omega \in \{0,1\}^{[\ell-1]\setminus\{1\}} }(\nu(n+\psi_1(\ve{s}^{(\omega)})) - 1) \\
 &
  \prod_{i=2}^{\ell-1} \prod_{\omega \in \{0,1\}^{[\ell-1]\setminus\{i\}}} \nu_i^{(\omega_1)}(n+\psi_i(\ve{s}^{(\omega)})) \prod_{i=\ell+1}^k \prod_{\omega \in \{0,1\}^{\ell-1}} f_i^{(\omega_1)}(n+\psi_i(\ve{s}^{(\omega)})) , 
\end{align*}
By the Cauchy-Schwarz inequality in the $s_{\ell}$ variable, we see that $I_{\ell-1}^2$ is bounded by the product of
\[ \E_{n\in G}   \E_{\substack{s_1^{(0)},\cdots,s_{\ell-1}^{(0)} \\ s_1^{(1)},\cdots,s_{\ell-1}^{(1)}}} \E_{s_{\ell+1},\cdots,s_k} \prod_{\omega \in\{0,1\}^{[\ell-1]}} \nu_{\ell}^{(\omega_1)}(n+\psi_{\ell}(\ve{s}^{(\omega)})) \]
and another term which, after expanding out the square, becomes exactly $I_{\ell}$. Since the expression above is $1+o(1)$ by the $k$-linear forms conditions on $\nu$, the desired claim follows.
\end{proof}

\subsection{Proof of \eqref{eq:counting2}}

Split the left side of \eqref{eq:counting2}  into the sum of two terms
\begin{equation}\label{eq:counting21} 
\E_{n\in G} (f_1'(n)-\widetilde{f}_1'(n)) ( f_1'(n)-\min(f_1'(n),1))+ \E_{n\in G}  (f_1'(n)-\widetilde{f}_1'(n)) ( \min(f_1'(n),1) - \widetilde{f}_1'(n)). 
\end{equation}
The first term above can be bounded by
\[ \E_{n\in G} (\nu_1'(n)+1)|\nu_1'(n)-1| \leq \E_{n\in G} |\nu_1'(n)-1|^2 + 2 \left( \E_{n\in G} |\nu_1'(n)-1|^2 \right)^{1/2}.  \]
This is $o(1)$ since the $L^2$-norm of $\nu_1'-1$ is $o(1)$ by the following lemma.

\begin{lemma}\label{lem:nu1'l2}
Let $\nu_1'$ be defined as above. Then 
\[ \E_{n\in G} |\nu_1'(n)-1|^2 = o(1). \]
\end{lemma}

\begin{proof}
It suffices to show that
\[ \E_{n\in G} \nu_1'(n)^2 = 1+o(1),\ \ \E_{n\in G}\nu_1'(n) = 1+o(1). \]
We prove only the first bound; the second bound is dealt with similarly and easier. Expanding the square we get
\[ \E_{n\in G} \nu_1'(n)^2 = \E_{n \in G} \E_{d^{(0)},d^{(1)}\in [D]} \prod_{i=2}^k \prod_{\tau\in\{0,1\}} \nu_i(n+(i-1)d^{(\tau)}). \]
For $s_2, \cdots, s_k \in [S]$, we may translate both  $d^{(0)}$ and $d^{(1)}$ by $s_2+\cdots+s_k$ with an error in the form of \eqref{eq:error} (with $f_i^{(\tau)}$ there replaced by $\nu_i$), which is $o(1)$. Thus
\[ \E_{n \in G} \nu_1'(n)^2 = \E_{n \in G} \E_{s_1^{(0)},s_1^{(1)} \in [D]} \E_{s_2,\cdots,s_k \in [S]} \prod_{i=2}^k \prod_{\tau \in \{0,1\}} \nu_i\left(n + (i-1)(s_1^{(\tau)}+s_2+\cdots+s_k) \right) + o(1). \]
After replacing $n$ by $n-s_2-2s_3-\cdots-(k-1)s_k$ we obtain
\[ \E_{n\in G}\nu_1'(n)^2 = \E_{n \in G} \E_{s_1^{(0)},s_1^{(1)}\in [D]} \E_{s_2,\cdots,s_k\in [S]} \prod_{i=2}^k \prod_{\tau\in\{0,1\}} \nu_i(n+\psi_i(\ve{s}^{(\tau)})) + o(1). \]
Since $\nu_i\in \{\nu,1\}$, this is $1+o(1)$ by the $k$-linear forms conditions on $\nu$, completing the proof of the lemma.
\end{proof}

It remains to bound the second sum in \eqref{eq:counting21}. First we claim that $\| \min(f_1',1) - \widetilde{f}_1' \|_{D,1} = o(1)$. To this end, let $u_2,\cdots,u_k:G\rightarrow [-1,1]$ by any functions, and define $u_1':G\rightarrow [-1,1]$ by
\[ u_1'(n)=\E_{d\in [D]} \prod_{i=2}^k u_i(n+(i-1)d). \]
By definition of the norm $\| \cdot \|_{D,1}$, we need to show that
\[ \Lambda_D(\min(f_1',1)-\widetilde{f}_1', u_2,\cdots,u_k) = o(1). \]
The left side above can be written as
\[  \E_{n\in G} \left[\min(f_1'(n),1)-\widetilde{f}_1'(n)\right] u_1'(n) \\
= \E_{n\in G} \left[\min(f_1'(n),1)-f_1'(n)\right] u_1'(n) + \E_{n\in G} (f_1'(n)-\widetilde{f}_1'(n)) u_1'(n).
\]
The first term can be bounded in terms of the $L^2$-norm of $\nu_1'-1$, which is $o(1)$ by Lemma~\ref{lem:nu1'l2}. The second term can be rewritten as
\[ \Lambda_D(u_1', f_2, \cdots, f_k) - \Lambda_D(u_1', \widetilde{f}_2,\cdots, \widetilde{f}_k). \]
This is $o(1)$ by the induction hypothesis (since $u_1'$ is bounded by $1$). This proves the claim.

Going back to the task of bounding the second sum in \eqref{eq:counting21}, we need to show that
\[ \E_{n\in G}  (f_1'(n)-\widetilde{f}_1'(n)) ( \min(f_1'(n),1) - \widetilde{f}_1'(n)) = o(1). \]
Expand it into four terms:
\[ \E_{n \in G} f_1'(n)\min(f_1'(n),1) - \E_{n \in G} f_1'(n)\widetilde{f}_1'(n) - \E_{n\in G} \widetilde{f}_1'(n)\min(f_1'(n),1) + \E_{n \in G}\widetilde{f}_1'(n)^2. \]
These terms can be rewritten as
\[ \Lambda_D(\min(f_1',1),f_2,\cdots,f_k) - \Lambda_D(\widetilde{f}_1', f_2, \cdots, f_k) - \Lambda_D(\min(f_1',1), \widetilde{f}_2,\cdots,\widetilde{f}_k) + \Lambda_D(\widetilde{f}_1',\widetilde{f}_2,\cdots,\widetilde{f}_k) . \]
Each of these four terms is $ \Lambda_D(\widetilde{f}_1',\widetilde{f}_2,\cdots,\widetilde{f}_k) + o(1)$ by the induction hypothesis, since both $\min(f_1',1)$ and $\widetilde{f}_1'$ are bounded by $1$ and $\|\min(f_1',1) - \widetilde{f}_1'\|_{D,1} = o(1)$ by the previous claim.  This proves \eqref{eq:counting2} and completes the proof of Proposition \ref{prop:counting}.

\bibliographystyle{plain}
\bibliography{narrow}{}

\begin{thebibliography}{10}

\bibitem{CFZ13}
D.~Conlon, J.~Fox, and Y.~Zhao.
\newblock A relative {S}zemer\'edi theorem.
\newblock {\em Geom. Funct. Anal.}, 25(3):733--762, 2015.

\bibitem{For07}
K.~Ford.
\newblock Simple proof of {G}allagher’s singular series sum estimate.
\newblock {\em arXiv preprint ArXiv:1108.3861}, 2007.

\bibitem{Gal76}
P.~X. Gallagher.
\newblock On the distribution of primes in short intervals.
\newblock {\em Mathematika}, 23(1):4--9, 1976.

\bibitem{GPY09}
D.~A. Goldston, J.~Pintz, and C.~Y. Y{\i}ld{\i}r{\i}m.
\newblock Primes in tuples. {I}.
\newblock {\em Ann. of Math. (2)}, 170(2):819--862, 2009.

\bibitem{GY07}
D.~A. Goldston and C.~Y. Y{\i}ld{\i}r{\i}m.
\newblock Higher correlations of divisor sums related to primes. {III}. {S}mall
  gaps between primes.
\newblock {\em Proc. Lond. Math. Soc. (3)}, 95(3):653--686, 2007.

\bibitem{Gow10}
W.~T. Gowers.
\newblock Decompositions, approximate structure, transference, and the
  {H}ahn-{B}anach theorem.
\newblock {\em Bull. Lond. Math. Soc.}, 42(4):573--606, 2010.

\bibitem{GT08}
B.~Green and T.~Tao.
\newblock The primes contain arbitrarily long arithmetic progressions.
\newblock {\em Ann. of Math. (2)}, 167(2):481--547, 2008.

\bibitem{GT10}
B.~Green and T.~Tao.
\newblock Linear equations in primes.
\newblock {\em Ann. of Math. (2)}, 171(3):1753--1850, 2010.

\bibitem{GTZ12}
B.~Green, T.~Tao, and T.~Ziegler.
\newblock An inverse theorem for the {G}owers {$U^{s+1}[N]$}-norm.
\newblock {\em Ann. of Math. (2)}, 176(2):1231--1372, 2012.

\bibitem{MR15}
K.~Matom{\"a}ki and M.~Radziwi{\l}{\l}.
\newblock Multiplicative functions in short intervals.
\newblock {\em arXiv preprint arXiv:1501.04585}, 2015.

\bibitem{MRT15}
K.~Matom{\"a}ki, M.~Radziwi{\l}{\l}, and T.~Tao.
\newblock An averaged form of {C}howla's conjecture.
\newblock {\em arXiv preprint arXiv:1503.05121}, 2015.

\bibitem{May15}
J.~Maynard.
\newblock Small gaps between primes.
\newblock {\em Ann. of Math. (2)}, 181(1):383--413, 2015.

\bibitem{RTTV08}
O.~Reingold, L.~Trevisan, M.~Tulsiani, and S.~Vadhan.
\newblock New proofs of the {G}reen-{T}ao-{Z}iegler dense model theorem: An
  exposition.
\newblock {\em arXiv preprint arXiv:0806.0381}, 2008.

\bibitem{Tao07}
T.~Tao.
\newblock A remark on {G}oldston-{Y}{\i}ld{\i}r{\i}m correlation estimates.
\newblock {\em Preprint}.

\bibitem{TZ08}
T.~Tao and T.~Ziegler.
\newblock The primes contain arbitrarily long polynomial progressions.
\newblock {\em Acta Math.}, 201(2):213--305, 2008.

\bibitem{TZ14}
T.~Tao and T.~Ziegler.
\newblock Narrow progressions in the primes.
\newblock {\em arXiv preprint arXiv:1409.1327}, 2014.

\bibitem{Zhang14}
Y.~Zhang.
\newblock Bounded gaps between primes.
\newblock {\em Ann. of Math. (2)}, 179(3):1121--1174, 2014.

\bibitem{Zhao14}
Y.~Zhao.
\newblock An arithmetic transference proof of a relative {S}zemer\'edi theorem.
\newblock {\em Math. Proc. Cambridge Philos. Soc.}, 156(2):255--261, 2014.

\end{thebibliography}

\end{document}